\documentclass{cedram-alco}

\usepackage{tikz}
\usetikzlibrary{shapes}

\title
[On the automorphism group of a Conway 99-graph]{On the automorphism group of a putative Conway 99-graph}
  
\author[\initial{P.} \middlename{G.} Cesarz]{\firstname{Patrick} \middlename{G.} \lastname{Cesarz}}

\address{Dept. of Mathematics and Statistics\\
University of Wyoming\\ 
Laramie\\
WY 82071 (USA)}
\email{pcesarz@uwyo.edu}

\author[\initial{A.} \middlename{J.} Woldar]{\firstname{Andrew} \middlename{J.} \lastname{Woldar}}

\address{Dept. of Mathematics and Statistics\\
Villanova University\\ 
Villanova\\
PA  19085 (USA)}
\email{andrew.woldar@villanova.edu}

\thanks{}

 

\begin{document}

\begin{abstract}
Let $\Gamma$ be a  {Conway 99-graph}, that is, a strongly regular graph   
 with parameters 
$(99,14,1,2)$.  In  
Makhnev and Minakova  (On automorphisms of strongly regular graphs with parameters $\lambda =1$, $\mu= 2$,  
Discrete Math.\  Appl.\ 
14 (2)  (2004) 201-210), 
 the authors prove  that   
the automorphism group 
$G$ of $\Gamma$ must have order dividing $2\cdot 3^3\cdot 7\cdot 11$. 
They further show that if $|G|$ is divisible by $2$ then $|G|$ must  divide $42$.   
In the present paper, we refine these results by proving that divisibility by $7$ implies  
$G \cong\mathbb Z_7$. As a consequence, divisibility by $2$ implies $|G|$ divides $6$, \ie $G$ is isomorphic to one of $\mathbb Z_2, \mathbb Z_6, S_3$.
\end{abstract}

\maketitle

\section{Introduction} 

The question of existence  of a strongly regular graph with parameters
$(99,14, 1,2)$ is a longstanding open problem. Its possible existence was first suggested by Norman Biggs  in 1969  \cite{Biggs}.  According to the  account given by Richard Guy in \cite{Guy}, John H.\ Conway worked on this problem as early as 1975. Later, Conway would offer a \$1,000 prize to anyone who could solve it (see \cite{Conway}, where it is listed as problem 2 among five posed open problems). From that point on, the graph came to be known colloquially as a \emph{Conway 
99-graph}.

In  \cite{Conway} Conway gave an alternate formulation of the existence problem, which we here reproduce: 
\begin{quote}
{Is there a graph with 99 vertices in which every edge (\ie pair of joined vertices)
belongs to a unique triangle and every nonedge (pair of unjoined vertices) to a unique
quadrilateral?}
\end{quote}
 
 Famously, A.A.\ Makhnev and I.M.\ Minakova proved in  \cite{Makhnev} that a
 strongly regular graph with parameters $(v, k, 1, 2)$ can only exist if 
 $k = u^2 + u + 2$ where $u\in\{1, 3, 4, 10, 31\}$.  Such graphs are known to exist for $u\in \{1,4\}$ but not much is known about the remaining cases. Should a Conway 99-graph exist, it would correspond to the case $u=3$.

The survey paper \cite{Makhnev2} by Makhnev neatly establishes a context for our work.  Let $G$ be the automorphism group of a putative Conway 99-graph. 
In \cite[Theorem 2.7]{Makhnev} it is stated that the order of $G$  must divide   $2\cdot 3^3\cdot 7\cdot 11$.  In  \cite[Corollary 2.4]{Makhnev2} Makhnev asserts the following: 

 If $G$ contains an involution $t$, then one of the following holds: 
  \begin{enumerate}
  \item
  If $|G|$ is divided by 7 then $|G|$ divides 42, $[O(G), t] = 1$, and in the case $|G| = 42$ the subgroup $O(G)$ is non-Abelian, 
\item $|G|$ divides 6.
  \end{enumerate}       
In the present paper, we strengthen these results as follows:

 \begin{enumerate}
 \setlength{\itemindent}
 {.4em}
 \item[($1^\prime$)]
  If $|G|$ is divisible by 7, then $G$ is isomorphic to  $\mathbb Z_7$.
     
  \item[($2^\prime$)]   
   If 2 divides $|G|$, then
 $|G|$ divides 6. 
\end{enumerate}

Our proof of ($1^\prime$) is performed in three stages, the first of which is to prove 14 cannot divide the order of $G$.  At the second stage, we show that divisibility by 7 implies
$G$ is isomorphic to $\mathbb Z_7$ or $Frob(21)$. Here $Frob(21)$ denotes the Frobenius group of order 21, \ie the index 2 subgroup of the holomorph $\mathbb Z_7\rtimes {\rm Aut}(\mathbb Z_7)\cong Frob(42)$.  The third stage is devoted to eliminating $Frob(21)$ as a possibility, and here we find it necessary to enlist the aid of a computer.  It is worthy of note that stage 3 requires garnering as much structural information  as possible in order to make a computer search feasible.

Our notation and terminology are standard. We refer the reader to 
\cite{Biggs1,BCN,Godsil,Rotman} as excellent sources of background material.

Our paper is organized as follows.  In Section \ref{prelims} we establish terminology and notation to be used throughout the paper.  Especially relevant to later sections is a labeling scheme we provide for the vertices of a putative Conway 99-graph $\Gamma$ under the assumption that 7 divides $|G|$.  This scheme allows us to embed the automorphism group $G$ of $\Gamma$ into the symmetric group of degree 14.  

In Section \ref{no order 14 auts} we prove there are no order 14 automorphisms of $\Gamma$.   In Section \ref{cyclic or Frobenius} we complete the proof of  ($2^\prime$) and show that divisibility by 7 implies $G\cong \mathbb Z_7$ or $Frob(21)$. 
Orbit valencies are determined in the case when
$G\cong Frob(21)$. 

In Section \ref{several results} we derive a fairly comprehensive structural framework for $\Gamma$ based on the assumption that $G\cong Frob(21)$. 
This framework is crucial in reducing run-time, thereby making a computer search feasible. Results of this search show that $G\cong\mathbb Z_7$  if $7$ divides $|G|$. Details of our program are provided  in Section \ref{computer}.  
  
\section{Preliminaries}\label{prelims}
Throughout this paper, $\Gamma$ will denote a putative Conway 99-graph, \ie a strongly regular graph with parameters $(99,14,1,2)$.  We denote its automorphism group  by $G$.  Although $\Gamma$ is not vertex transitive, one can always ``hang'' the graph from any vertex $x\in V(\Gamma)$ whereby vertices are grouped together in accordance with their distance from $x$. 
This is a property of  distance-regular graphs in general and strongly regular graphs in particular. In such case, we refer to $x$ as the ``root vertex''. 

As is customary, we denote by $\Gamma_1\,(x)$ and $\Gamma_2\,(x)$ the set of neighbors and non-neighbors of $x$ respectively, commonly referred to as the first and second subconstituents of $\Gamma$.  When $x$ is understood from context, we will abbreviate these sets by $\Gamma_1$ and $\Gamma_2$. In Figure \ref{DDD} we depict the distance distribution diagram for $\Gamma$. Note that the diagram  indicates that $|\Gamma_1|=14$ and $|\Gamma_2|=84$.

For the moment we assume $x\in V(\Gamma)$ is an arbitrary vertex, however in due course our choice of $x$ will carry special significance.  
We label the 14 neighbors of $x$ by ${iX}$, $1\le i\le 7$, $X\in \{L,R\}$, see Figure \ref{Gamma1}.  Since $\lambda=1$, we may assume without loss of generality that 
$\{iL,iR\}$ is an edge for every $1\le i\le 7$.

\begin{figure}
\begin{tikzpicture}
\draw (0,-1) circle (8pt);
\draw (2.5,-1) circle (8pt);
\draw (5,-1) circle (8pt); 
\draw (0,-1)--(2.5,-1)--(5,-1);
\draw[white, fill] (0,-1) circle (7pt); 
\draw[white, fill] (2.5,-1) circle (7pt); 
\draw[white, fill] (5,-1) circle (7pt); 
\node at (0,-1) {1} node at (.55,-.83) {\footnotesize 14};
\node at (2.5,-1) {14} node at (2.03,-.83) {\footnotesize 1} node at (3.05,-.83) {\footnotesize 12} node at (2.5,-.53) {\footnotesize 1};
\node at (5,-1) {84} node at (4.55,-.83) {\footnotesize 2} node at (5,-.53) {\footnotesize 12};

\end{tikzpicture}
\caption{Distance distribution diagram of $\Gamma$}\label{DDD}
\end{figure}
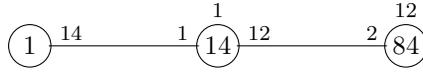

\begin{center}
\begin{figure}
\begin{tikzpicture}[scale=.9]

\draw [fill] (0,0) circle (.06) node [above] {$x$};

\draw [fill] (-.5,-2) circle (.06) node at (-.5,-2.3) {${4L}$};
\draw [fill] (.5,-2) circle (.06) node at (.5,-2.3) {${4R}$};

\draw [fill] (-1.5,-2) circle (.06) node at (-1.5,-2.3) {${3R}$};;
\draw [fill] (1.5,-2) circle (.06) node at (1.5,-2.3) {${5L}$};;
\draw [fill] (-2.5,-2) circle (.06) node at (-2.5,-2.3) {${3L}$};;
\draw [fill] (2.5,-2) circle (.06)  node at (2.5,-2.3) {${5R}$};;
\draw [fill] (-3.5,-2) circle (.06) node at (-3.5,-2.3) {${2R}$};;
\draw [fill] (3.5,-2) circle (.06) node at (3.5,-2.3) {${6L}$};;
\draw [fill] (-4.5,-2) circle (.06) node at (-4.5,-2.3) {${2L}$};;
\draw [fill] (4.5,-2) circle (.06) node at (4.5,-2.3) {${6R}$};;
\draw [fill] (-5.5,-2) circle (.06) node at (-5.5,-2.3) {${1R}$};;
\draw [fill] (5.5,-2) circle (.06) node at (5.5,-2.3) {${7L}$};;
\draw [fill] (-6.5,-2) circle (.06) node at (-6.5,-2.3) {${1L}$};;
\draw [fill] (6.5,-2) circle (.06) node at (6.5,-2.3) {${7R}$};;

\draw (-.5,-2) -- (.5,-2);
\draw (-1.5,-2) -- (-2.5,-2);
\draw (-3.5,-2) -- (-4.5,-2);
\draw (-5.5,-2) -- (-6.5,-2);
\draw (1.5,-2) -- (2.5,-2);
\draw (3.5,-2) -- (4.5,-2);
\draw (5.5,-2) -- (6.5,-2);

\draw (0,0) -- (.5,-2);
\draw (0,0) -- (-2.5,-2);
\draw (0,0) -- (-4.5,-2);
\draw (0,0) -- (-6.5,-2);
\draw (0,0) -- (2.5,-2);
\draw (0,0) -- (4.5,-2);
\draw (0,0) -- (6.5,-2);

\draw (0,0) -- (-.5,-2);
\draw (-1.5,-2) -- (0,0);
\draw (-3.5,-2) -- (0,0);
\draw (-5.5,-2) -- (0,0);
\draw (1.5,-2) -- (0,0);
\draw (3.5,-2) -- (0,0);
\draw (5.5,-2) -- (0,0);

\end{tikzpicture}
\caption{A labeling scheme for $V(\Gamma_1)$}\label{Gamma1}
\end{figure}
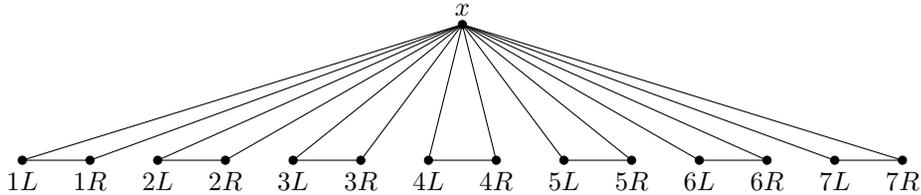
\end{center}

Observe that for any $X,Y\in \{L,R\}$, the vertices ${iX},{jY}\in V(\Gamma_1)$ are nonadjacent if and only if
$i\ne j$.  
As $\mu=2$, each pair of nonadjacent vertices ${iX},{jY}\in  V(\Gamma_1)$ must have a unique common neighbor in $V(\Gamma_2)$. We label this common neighbor  $ijXY$.  As there are $\binom{14}{2} -7=84$ pairs of nonadjacent vertices in $V(\Gamma_1)$, we see that each of the 84 vertices in $V(\Gamma_2)$ receives a unique label, again due to $\mu=2$. 
 
\begin{lemm}\label{identity aut}
Suppose $g \in G$ fixes the subgraph $\Gamma_1+x$ pointwise. Then $g$ fixes 
$\Gamma_2$ pointwise, i.e, $g$ is the identity automorphism.
\end{lemm}

\begin{proof}
Evident from our labeling scheme  
and the fact that $\mu=2$.
\end{proof}

\begin{lemm}\label{7-orbits}
Suppose there exists $s\in G$ with $|s| = 7$.  Then 
$s$ fixes a unique vertex in $V(\Gamma)$, hence the $\langle s\rangle$-orbit structure 
  on $\Gamma$ is $[1,7^{14}]$.   
\end{lemm}

\begin{proof}
As $|V(\Gamma)| = 99 \equiv 1 \!\pmod 7$ we deduce that $s$ must have at least one fixed vertex which we may choose to fulfill the role  of root vertex $x$ of $\Gamma$. This establishes that both  
$\Gamma_1$ and $\Gamma_2$ are $s$-invariant. In particular, the orbit structure of $\langle s\rangle$ on $\Gamma_1$ is $[7^2]$, $[1^7,7]$, or $[1^{14}]$.   
As $iL$ is a fixed point of $s$ if and only if $iR$ is, the number of fixed points must be even, \ie the $\langle s\rangle$-orbit structure  is either 
$[7^2]$ or $[1^{14}]$. However, Lemma \ref{identity aut} rules out $[1^{14}]$, leaving  $[7^2]$ as 
the orbit structure on $\Gamma_1$. Now suppose a vertex $ijXY\in V(\Gamma_2)$ is fixed by $s$.  Then $x$ and $ijXY$ have  ${iX},{jY},{iX}^s, {jY}^s$ as common neighbors.  As $i\ne j$ and $\mu =2$,  it follows that $\{{iX},{jY}\}=\{{iX}^s, {jY}^s\}$. Since $s$ has no fixed points on $\Gamma_1$, this implies ${iX}^s={jY}$ and ${jY}^s=iX$. But then
each of $iX$ and $jY$ is fixed by $s^2$, a contradiction.
  The result follows.  
\end{proof}

\begin{rema}\label{|s|=7} 
Note that one consequence of our labeling scheme, together with the assumption in Lemma \ref{7-orbits},  
is that the automorphism group ${\rm Aut}(\Gamma)$ embeds in ${\rm Sym}(V(\Gamma_1))\cong S_{14}$.  However, among all order $7$ elements in 
${\rm Sym}(V(\Gamma_1))$, the only ones that are graph automorphisms of $\Gamma_1$ are powers of $(1L,2L,\dots,7L)\,(1R,2R,\dots,7R)$.
Henceforth we denote $s_L=(1L,2L,\dots,7L)$, $s_R=(1R,2R,\dots,7R)$ and $s=s_Ls_R$. 
The two $\langle s\rangle$-orbits on $\Gamma_1$ are now transparent. They are 
 $\{1L,2L,\dots,7L\}$ and $\{1R,2R,\dots,7R\}$. 
\end{rema}

\begin{lemm}\label{140 triangles}  
Every vertex of $\,\Gamma_2$ lies on five $3$-cycles wholly inside $\Gamma_2$. Thus 
there are  $140$ $3$-cycles in $\Gamma_2$.
\end{lemm}

\begin{proof}
Clearly, every vertex in $\Gamma$ lies on seven $3$-cycles. Given a vertex $v\in V(\Gamma_2)$ it has precisely two $\Gamma_1$-neighbors $u$ and $w$.  As 
$\lambda=1$, each of $vu$ and $vw$ must be an edge in a unique 3-cycle. Moreover, these two 3-cycles cannot coincide. Indeed, this would require that $u$ and $w$ be adjacent, whence $uwx$ would be a second 3-cycle on the edge $uw$ where $x$ is the root vertex.  This proves the remaining five $3$-cycles on $v$ lie entirely inside $\Gamma_2$. But now we have that the total number of $3$-cycles in $\Gamma_2$ is $\frac{84\cdot 5}{3}=140$ as claimed.
\end{proof}

\section{Nonexistence of an order 14 automorphism}\label{no order 14 auts}
 
Our goal in this section is to 
prove $G$ does not contain any element of order 14. This will be a crucial step in our argument  that  divisibility by 2 implies $|G|$  divides $6$.  

Lemmas \ref{order 14 element} and \ref{orbits} set the groundwork for the rest of this section.   

\begin{lemm}\label{order 14 element}
Suppose $G$ contains a cyclic subgroup $K$ of order $14$. 
Then 
$K=\langle st\rangle$ where $s=(1L,2L,\dots,7L)\,(1R,2R,\dots,7R)$ and $t=(1L,1R)\,(2L,2R)\,\dots\,(7L,7R)$. 
\end{lemm}

\begin{proof} Let $C_S(s)=(\langle s_L\rangle\times \langle s_R\rangle)\rtimes\langle t
\rangle$ denote the centralizer of $s$ in $S={\rm Sym}(V(\Gamma_1))$.   There are seven involutions in $C_S(s)$ which take the form 
\[t^{\,s_{R}^{\,i}}=(1L,(1+i)R)\,(2L,(2+i)R)\dots(7L,iR)\]  
for $0\le i\le 6$, and in each case $\langle st^{\,s_{R}^i}\rangle$ is a cyclic subgroup of $C_S(s)$ of order 14.     
However,  since $t^{\,s_{R}^{\,i}}$ maps $1L$ to $(1+i)R$ and $1R$ to $(1-i)L$, adjacency is preserved only if $i=0$.   
  Thus $t$ is the unique involution in ${\rm Aut}(\Gamma_1)$ whereby   
 $K=\langle st\rangle$ is the unique cyclic subgroup  of order 14  in ${\rm Aut}(\Gamma_1)$.  
\end{proof}

\begin{lemm}\label{orbits}
$K =\langle st \rangle$ fixes the root vertex $x$ of $\,\Gamma$ but has no other fixed points. Thus each of the remaining seven $K$-orbits on $\,\Gamma$ has size $14$.  
\end{lemm}

\begin{proof}
Recall from Lemma \ref{7-orbits} that $x$ is the unique vertex fixed by $s\in K$. As $t$ commutes with $s$
we have that $x^t$ is fixed by $s$, whence $x^t=x$. Thus $x$ is the unique vertex fixed by $K$. 

We now consider $K$-orbits on $\Gamma -x$. 
 Clearly, $K$ fuses the two $\langle s\rangle$-orbits in $\Gamma_1$ so we are left to consider the orbit structure of $K$ on $\Gamma_2$. Since every $\langle s\rangle$-orbit on $\Gamma_2$ has size 7, the only possible size of  a $K$-orbit on $\Gamma_2$ is 7 or 14. However, $t$ cannot fix any vertex $ijXY\in V(\Gamma_2)$ where $X,Y\in\{L,R\}$.  Indeed, this would  imply $ijXY=(ijXY)^t = ijX^CY^C$ where $\{L,R\}=\{X,X^C\}=\{Y,Y^C\}$. If $X\ne Y$, then $i=j$ which violates $\lambda=1$. Otherwise $X=Y$, which implies   
 $ijXX=(ijXX)^t=ijX^CX^C$, a contradiction since $X\ne X^C$. 
 \end{proof}

\begin{rema}\label{rem: coords}
It is easy to see that a set of orbit representatives in the action of $K$ on $\Gamma_2$ is given by  
$\{12LL,\, 13LL,\,14LL,\, 12LR,\, 13LR,\, 14LR\}$.  Moreover, each numerical coordinate $i$ occurs exactly four times in each orbit. For example, in the orbit with representative $12LL$, the coordinate $3$ occurs in each of $23LL$,  
 $34LL$, $23RR$, $34RR$. However, these four vertices are distributed evenly into pairs in the sense that $23LL$, $34LL$ are neighbors of $3L$ while $23RR$, $34RR$ are neighbors of $3R$. 
\end{rema}

Consider the equitable partition $\pi$ induced by the $K$-orbits $\mathcal O_1,\mathcal O_2,\dots,\mathcal O_6$ on $\Gamma_2$. As is customary, we shall refer to $\pi$ as an \emph{orbit partition}.  We denote by $b_{ij}$ the number of $\mathcal O_j$-neighbors of any fixed vertex in $\mathcal O_i$.  When $i=j$ we simply write $b_i$   and refer to it as the \emph{internal valency} of the orbit $\mathcal O_i$.  We shall also call the $b_{ij}$ \emph{orbit valencies} (or simply \emph{valencies}) due to what occurs naturally in the quotient graph $\Gamma_2/\pi$.  Pictorially, $b_{ij}$ appears as a label of an arc from $\mathcal O_i$ to $O_j$, however in our case this arc is an edge (\ie  $b_{ij}=b_{ji}$ for all $i,j$) since all orbits $\mathcal O_i$ have the same size.

 \begin{figure}
 \resizebox{9cm}{7.5cm}{
\begin{tikzpicture}
\draw[thick] (0,0) circle (6mm); 
\draw [thick](3,0) circle (6mm);
\draw [thick](0,3) circle (6mm)  node at (0,3) {$b_{3}$};
\draw[thick] (3,3) circle (6mm)  node at (3,3) {$b_{4}$};
\draw[thick] (-2.5,1.5) circle (6mm)  node at (-2.5,1.5) {$b_{5}$};
\draw[thick] (5.5,1.5) circle (6mm) node at (5.5,1.5) {$b_{6}$};

\draw (.6,0)--(2.4,0);
\draw (.6,3)--(2.4,3);
\draw (0,.6)--(0,2.4);
\draw (3,.6)--(3,2.4);
\draw (.4,.4)--(2.6,2.6);
\draw (.4,2.6)--(2.6,.4);

\node at (0,0) {$b_{1}$};
\node at (1.5,.238) {$b_{12}$};
\node at (.75,1.2) {$b_{14}$};
\node at (-.278,1.5) {$b_{13}$}; 
\node at (-1.2,.9) {$b_{15}$};
\node at (5,.095) {$b_{16}$}; 

\node at (3,0) {$b_{2}$};
\node at (3.35,1.5) {$b_{24}$}; 
\node at (2.2,1.2) {$b_{23}$};
\node at (-2,.095) {$b_{25}$}; 
\node at (4.2,.9) {$b_{26}$};

\node at (1.5,3.225) {$b_{34}$};
\node at (1.5,5.05) {$b_{56}$};

\node at (-1.95,3.1) {$b_{45}$}; 
\node at (4.95,3.1) {$b_{36}$};

\node at (-1.2,2.1) {$b_{35}$};
\node at (4.2,2.1) {$b_{46}$};

\node at (-.8,-.45) {\Large$\mathcal O_1$};
\node at (3.85,-.45) {\Large$\mathcal O_2$};
\node at (-.82,3.5) {\Large$\mathcal O_3$};
\node at (3.85,3.5) {\Large$\mathcal O_4$};

\node at (-3.45,1.5) {\Large$\mathcal O_5$};
\node at (6.47,1.5) {\Large$\mathcal O_6$};

\draw (-2,1.1)--(-.55,.17);
\draw (-2,1.9)--(-.55,2.83);

\draw (5,1.1)--(3.55,.17);
\draw (5,1.9)--(3.55,2.83);

\draw [bend right=80] (2.8,3.6) to (-2.6,2.1); 
\draw [bend left=80] (2.8,-.6) to (-2.6,.9); 

\draw [bend right=80] (0,-.6) to (5.5,.9); 
\draw [bend left=80] (0,3.6) to (5.5,2.1); 

\draw (-2.6,2.1) to[bend left=90, looseness=1.35] (5.5,2.1);

\end{tikzpicture}}
\vspace{-7mm}
   \caption{A general $K$-orbit partition  on $\Gamma_2$ where $K=\langle st\rangle\cong \mathbb Z_{14}$}\label{fig:6 orbits}  
\end{figure}
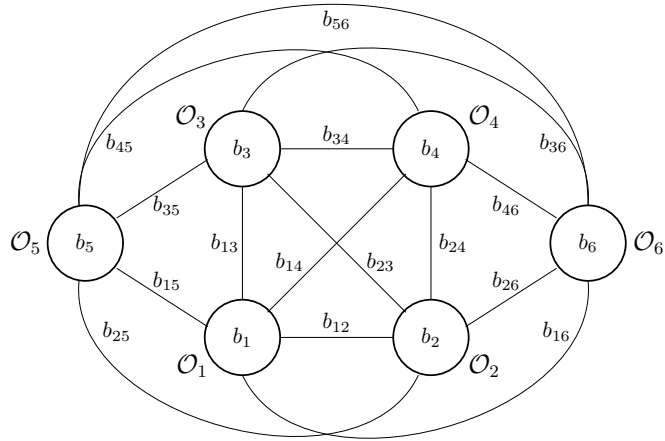

At present we have that $st$ is an order 14 automorphism of the subgraph $\Gamma_1+x$.  We wish to show $st$ cannot extend to an automorphism of 
$\Gamma$.   Our first step toward this objective is to count  in two ways the cardinality of the set 
\[S = \{uvw: \text{$uvw$ is a 2-path with } w \in \mathcal{O}_i\},\] 
where $u$ is a fixed vertex in $\mathcal O_i$ (see Figure \ref{fig:6 orbits}).

		For each of the $b_{i}$ neighbors of $u$ in $\mathcal O_i$ there is a unique $2$-path from $u$ to $w$ (since $\lambda=1$). Similarly, for each of the $13-b_{i}$ non-neighbors of $u$ in $\mathcal O_i$ there are two 2-paths from $u$ to $w$
  (since $\mu=2$).
 Thus $|S|= b_{i} \cdot 1 + (13 - b_{i}) \cdot 2 = 26 - b_{i}$.  
 
 On the other hand, we may condition our count on the location  of the intermediate vertex $v$.  For $v$ in $\mathcal O_j$ there are $b_{ij}\,(b_{ij}-1)$ such 2-paths. (Here and hereafter, we shall identify $b_{ii}$ with $b_i$ for notational convenience.) 
 In addition, $u$ has exactly two neighbors in $\Gamma_1$  each of which has a unique neighbor $w\in \mathcal O_i\setminus \{u\}$ (\cf Remark \ref{rem: coords}). This gives
 $|S|= \sum_{j= 1}^{6} b_{ij} \,(b_{ij} - 1) + 2$.  
 Equating these two expressions for $|S|$, we obtain   
$26 - b_{i} = \sum_{j = 1}^{6} b_{ij} \,(b_{ij} - 1) + 2$.  But due to the fact that  $\sum_{i=1}^6 b_{ij}=12$,  this simplifies to
\begin{equation}36 - \big(b_{i}^{\,2} + b_{i} \big) = \sum_{j = 2}^{6} b_{ij}^{\,2}\end{equation}\label{eqn1}
We divide our analysis into cases based on an assumed value for $b_{i}$. Once a choice of $b_i$ is made, we find all ways of expressing $36 - \big(b_{i}^{\,2 }+ b_{i} \big)$ as a sum of five squares while maintaining the valency requirement $\sum_{i=1}^6 b_{ij}=12$.  Note that $b_i\le 5$ since otherwise the value of $36 - \big(b_{i}^{\,2} + b_{i} \big)$ would be negative. 

All solution sets are provided in the lemma below. Verification of the list is   straightforward, so is left to the reader. 
Note that each solution set is expressed as an ordered pair of the form \[\big(b_i,\big\{a_2,a_3,a_4,a_5,a_6\big\}\big).\] This is because unlike the internal valency $b_i$ which remains fixed, the values $a_2,a_3,\dots,a_6$ may be assigned to the valencies $b_{ij}$, $j\ne i$, in any specified manner. Thus, there are multiple solutions corresponding to each solution set achieved by suitably permuting the members of the multiset  $\big\{a_2,a_3,a_4,a_5,a_6\big\}$. Below, we list these members in decreasing order.

\begin{lemm}\label{orbit type} For $b_{i}\in\{1,3,5\}$ there are no solutions to formula {\rm (1)}. For other values of $\,b_i$ the solutions are listed   as follows: 
\begin{enumerate}
\item[\rm (a)] \,  $\big(0,\big\{4, 3, 3, 1, 1\big\}\big)$ \,and \, 
			$\big(0,\big\{3, 3, 3, 3, 0\big\}\big)$\, when\, $b_{i}=0$.
\item[\rm (b)]\,  $\big(2, \big\{4, 3, 2, 1, 0\big\}\big)$\, when\, $b_{i}=2$,
\item[\rm (c)]\,
$\big(4, \big\{3, 2, 1, 1, 1\big\}\big)$\, and\,  $\big(4, \big\{2, 2, 2, 2, 0\big\}\big)$ \,when\, $b_{i}=4$.
\end{enumerate}
\end{lemm}
For future reference, it is convenient to designate these solution sets by type, e.g.	
 \[{\rm I}.\; \big(0, \big\{4, 3, 3, 1, 1\big\}\big),\;\; {\rm II}.\; \big(0, \big\{3, 3, 3, 3, 0\big\}\big),\;\; {\rm III}.\; \big(2, \big\{4, 3, 2, 1, 0\big\}\big)\]
 \[{\rm IV}.\; \big(4, \big\{3, 2, 1, 1, 1\big\}\big),\;\; {\rm V}.\;  \big(4, \big\{2, 2, 2, 2, 0\big\}\big)\]
 
 We also extend this terminology to orbits, saying an orbit is \emph{of type} 
 $\rm T$ if its set of valencies correspond to a solution set of type ${\rm T}$, ${\rm T}\in\{{\rm I, II, III, IV, V}\}$.

 We next count in two ways the number of $2$-paths starting from a fixed vertex 
 $u$ in $\mathcal{O}_i$ and ending at some vertex $w$ in $\mathcal{O}_j$, $j\ne i$.  Here $u$ has exactly $b_{ij}$ neighbors in $\mathcal{O}_j$, and as $\lambda=1$ there exists a unique $2$-path starting at $u$ and ending at $w$ for each neighbor $w$ of $u$ in $\mathcal O_j$.  Similarly, $u$ has  
$14 - b_{ij}$ non-neighbors in  $\mathcal{O}_j$, and as  $\mu=2$ there are exactly two $2$-paths starting at $u$ and ending at each non-neighbor $w$ of $u$ in $\mathcal O_j$.   Thus in total there are $b_{ij} \cdot 1 + (14 - b_{ij}) \cdot 2 = 28 - b_{ij}$ such $2$-paths from $u$ into $\mathcal O_j$ when $j\ne i$.  
		
For the second count, we focus on the location of an intermediate vertex $v$ in each such 2-path. Here $v$ can occur in any of the six $K$-orbits on $\Gamma_2$ as well as in $\Gamma_1$.  In the case of $K$-orbits on $\Gamma_2$, there are $b_{ik}$ choices for $v$ in 
$\mathcal O_k$, and for each such $v$ there are $b_{kj}$ choices for $w$ in $\mathcal O_j$.  This gives $b_{ik}b_{kj}$ 2-paths of desired type. 
In addition, there are two choices for  $v$ in  $\Gamma_1$ each of which has two neighbors $w$ in 
$\mathcal O_j$. This produces four more 2-paths.  Thus, in total there are precisely  
 $\sum_{k= 1}^{6}b_{ik}b_{kj}  + 4$ paths of the type in question when $j\ne i$.

Equating these two counts yields $28 - b_{ij} = \sum_{k = 1}^{6} b_{ik}b_{kj} + 4$, or 
equivalently
\begin{equation}
24 - b_{ij} = \sum_{k = 1}^{6} b_{ik}b_{kj}.
\end{equation}

\begin{lemm}\label{orbit type1}
In a $K$-orbit partition of $\,\Gamma_2$ {\rm (}\cf Figure \ref{fig:6 orbits}{\rm )} we have the following:
\begin{enumerate}
\item[\rm (a)] The number of orbits of type I is at most $2$.
\item[\rm (b)]  The number of orbits of type II is at most $1$.
\item[\rm (c)]  The number of orbits of type III is at most $4$.
\item[\rm (d)]  The number of orbits of type IV is at most $4$.
\item[\rm (e)]  The number of orbits of type V is at most $1$.
\end{enumerate}
\end{lemm}

\begin{proof}
(a) Suppose there exist two orbits $\mathcal O_i$ and $\mathcal O_j$ of type I.  Then since $b_i=b_j=0$,  formula (2) reduces to  
$24 - b_{ij} =\sum_{k \ne i,j} b_{ik}b_{kj}$. 
Note that this formula is satisfied only if $b_{ij}=4$, which results in the solution  $20=3^2+3^2+1^2+1^2$.  
As an orbit of type I admits only one edge of valency 4,  there cannot be a third orbit of this type. \\[1mm]
(b) Let  $\mathcal O_i$ and $\mathcal O_j$ be two orbits of type II. 
Since $b_i=b_j=0$,  formula (2) again reduces to $24 - b_{ij} =\sum_{k \ne i,j} b_{ik}b_{kj}$. But regardless of how one chooses $b_{ij}\in \{0,3\}$ and reorders the corresponding  multiset,  this formula is never satisfied.  Thus there is at most one orbit of type II.\\[1mm]
(c) Let  $\mathcal O_i$ and $\mathcal O_j$ be two orbits of type III. 
Since $b_i=b_j=2$, formula (2) becomes $24 - 5b_{ij} =\sum_{k \ne i,j} b_{ik}b_{kj}$. In this case one has $b_{ij}\in\{0,1,2,3,4\}$, however there are no solutions if 
$b_{ij}\in\{1,2\}$. In contrast, every remaining choice of $b_{ij}$ works. 
Specifically, if $b_{ij}=0$ one gets $24=4\cdot 1+3\cdot 4+2\cdot 3+1\cdot 2$ as a solution. 
For  $b_{ij}=3$ one gets $9=4\cdot 0+2\cdot 4+1\cdot 1+0\cdot 2$, while for 
$b_{ij}=4$  one gets $4=3\cdot 0+2\cdot 1+1\cdot 2+0\cdot 3$. Having only three allowable valencies for edges between pairs of type III orbits, it is not possible to have a fifth orbit of this type.
 \\[1mm]
(d) Given any two orbits $\mathcal O_i$ and $\mathcal O_j$ of type IV, we have that $b_{ij}\in\{1,2,3\}$. As $b_1=b_2=4$, formula (2) becomes 
$24 -9b_{ij} =\sum_{k \ne i,j} b_{ik}b_{kj}$. Clearly $b_{ij}=1$ leads to a solution, namely  $15=1^2+1^2+2^2+3^2$, but other choices of $b_{ij}$ fail. Since a type IV orbit has only  three edges of valency 1, there can be at most four orbits of this type. 
\\[1mm]
(e) Suppose there are two orbits $\mathcal O_i$ and $\mathcal O_j$ of type V. 
 Then since $b_i=b_j=4$,  formula (2) becomes $24-9b_{ij}=\sum_{k \ne i,j} b_{ik}b_{kj}$. Here $b_{ij}\in\{0,2\}$, but it is immediate that neither choice leads to a solution.  This proves there is at most one orbit of type V.
 \end{proof}

In the above, we applied formula (2) to bound orbits of identical type in a $K$-orbit partition of $\Gamma_2$.  We now do the same for orbits of mixed type. Note that we do not strive to obtain sharp bounds at this stage. Our goal is simply to eliminate several possibilities in an expedient manner.   

 \begin{lemm}\label{mixed types}
 Let  $\rm T_i\,, T_j\in\{I,II,III,IV,V\}$. Then the $\rm (T_i\,,T_j)$-entry in Table \ref{table} bounds from above  the number of orbits of type $\rm T_j$ that can coexist with a fixed orbit of  type $\rm T_i$ in a $K$-orbit partition of $\,\Gamma_2$.  
 \end{lemm}

\begin{proof}
Note that the diagonal entries in Table \ref{table} were previously confirmed in Lemma \ref{orbit type1}. Moreover, one need not check any entry $\rm(T_i\,,T_j)$ that is equal to the diagonal entry $\rm(T_j\,,T_j)$ since the latter is the maximum allowable 
number of orbits of type $\rm T_j$ in any $K$-orbit partition of $\Gamma_2$.   

{\large\begin{table}[th]
\begin{tabular}{l||ccccc}
&I&II&III&IV&V\\ \hline\hline\\[-4mm]
I   & 2 & 0 & 4 & 0 & 0\\[.5mm]
II  & 0 & 1 & 4 & 4 & 1\\[.5mm]
III & 2 & 1 & 4 & 2 & 1\\[.5mm]
IV & 0 & 1& 4 & 4 & 1\\[.5mm]
V  & 0 & 1 & 4 & 4 & 1
\end{tabular}\\[3mm]
\caption{Bounds on the number of orbits of mixed type}\label{table}  
\end{table}}

\noindent{\bf Case 1. $\rm (I,\,II)=(II,\,I)=0$\,}:
Clearly, the only option for the shared edge is $b_{ij}=3$. As $b_i=b_j=0$, formula (2) reduces to $21 =\sum_{k\ne i,j}b_{ik}b_{kj}$ where $b_{ik}\in\{4,3,1,1\}$ and $b_{kj}\in\{3,3,3,0\}$. It is easy to see that no permutation of multisets leads to a solution, \ie $\rm (I,\,II)=0$. (For $\rm (II,\,I)=0$, the only change to the above is $b_{ik}\in\{3,3,3,0\}$ and $b_{kj}\in\{4,3,1,1\}$.)\\[2mm]
 \noindent{\bf Case 2.\ $\rm (I,\,IV)=(IV,\,I)=0$\,}:
In this case $b_i=0$ and $b_j=4$, so formula (2) reduces to 
$21 -5b_{ij}=\sum_{k\ne i,j}b_{ik}b_{kj}$. Here
there are two options for $b_{i,j}$.  If $b_{i,j}=1$ then we get
$16=\sum_{k\ne i,j}b_{ik}b_{kj}$ where $b_{ik}\in \{4,3,3,1\}$ and $b_{kj}\in\{3,2,1,1\}$, and no permutation of multisets leads to a solution. 
For the second option $b_{ij}=3$, we get $6=\sum_{k\ne i,j}b_{ik}b_{kj}$ where $b_{ik}\in \{4,3,1,1\}$ and $b_{kj}\in\{2,1,1,1\}$. Again no permutation of multisets gives a solution.  Thus
$\rm (I,\,IV)=(IV,I)=0$.\\[2mm]
 \noindent{\bf Case 3.\ $\rm (I,\,V)= (V,\,I)=0$\,}: Here the multisets   
$\{4,3,3,1,1\}$ and $\{2,2,2,2,0\}$ are disjoint, so there is no possible choice of valency for the edge between two orbits of these respective types. The result follows at once.\\[2mm] 
\noindent{\bf Case 4.\ $\rm (III,\,IV)=2$\,}: In this case, $b_i=2$ and $b_j=4$ so formula (2) becomes $24-7b_{ij}=\sum_{k\ne i,j}b_{ik}b_{kj}$. There are three choices for $b_{ij}$, namely $b_{ij}\in\{3,2,1\}$. If $b_{ij}=3$ the formula reduces to  $3 =\sum_{k\ne i,j}b_{ik}b_{kj}$ where $b_{ik}\in \{4,2,1,0\}$ and $b_{kj}\in\{2,1,1,1\}$, and it is immediate that there is no solution. If $b_{ij}=2$ we obtain $10=\sum_{k\ne i,j}b_{ik}b_{kj}$ where $b_{ik}\in \{4,3,1,0\}$ and $b_{kj}\in\{3,1,1,1\}$. Here there  is a unique solution, namely $10=4\cdot 1+3\cdot 1 +1\cdot 3+0\cdot 1$.  Finally, if $b_{ij}=1$ we obtain $17=\sum_{k\ne i,j}b_{ik}b_{kj}$ where $b_{ik}\in \{4,3,2,0\}$ and $b_{kj}\in\{3,2,1,1\}$. In this case there are three solutions, namely $17=4\cdot 3+3\cdot 1 +2\cdot 1 +0\cdot 2=
4\cdot 1+3\cdot 3 +2\cdot 2 +0\cdot 1= 
4\cdot 2+3\cdot 1 +2\cdot 3 +0\cdot 1$.
In any case, there are just two possibilities for the valency of an edge from a fixed type III orbit to an orbit of type IV.  We conclude that $\rm (III,\,IV)=2$.\\[2mm] 
As all cases in the lemma statement have been treated, the proof is complete. 
\end{proof}

The reader will note that  the relation in Lemma \ref{mixed types} is not generally symmetric.

 Let us write $\rm [I^{a},II^{\,b},III^{\,c},IV^{\,d},V^{\,e}]$ to indicate a $K$-orbit partition of $\Gamma_2$ having $a$ orbits of type I, $b$ orbits of type II, and so on.
(If an orbit of specific type does not occur in the partition, we simply omit that type from the above partition notation.)    
 
\begin{lemm}\label{no III,IV}
There is no $K$-orbit partition of the form    
$\rm [III^{\,c}, IV^{\,6-c}]$ for any $c$. 
\end{lemm}

\begin{proof}
 By Lemma \ref{orbit type1}, one has $2\le c\le 4$. As a type III orbit has a single edge of valency 4 and a type IV orbit has none, there are $c/2$ edges of valency 4 in the partition. This means   $c$ must be even.  However, $c= 2$ is prohibited. Indeed, by Lemma \ref{mixed types} the existence of a type III orbit requires that there be at most two type IV orbits, hence
 $6-c\le 2$. 
 
 We come now to  the only remaining case which is $c=4$.     
  As shown in the proof of Lemma \ref{orbit type1}, each pair of type III orbits must share an edge of valency 0, 3 or 4.  But as there are four
 type III orbits, every such valency gets used.  On the other hand, we showed in Lemma \ref{orbit type1} that  two orbits of type IV must share an edge of valency 1. This leaves a type IV orbit with  an unusable edge of valency 3, again a contradiction.   
\end{proof}

One conclusion of Lemma \ref{no III,IV} is that a viable $K$-orbit partition of $\Gamma_2$ must contain an orbit of type I, II or V. By way of the next two lemmas, we are able to narrow this down considerably.

\begin{lemm}\label{no type I}
There is no $K$-orbit partition that contains an orbit of type {\rm I}. 
\end{lemm}

\begin{proof}
By Lemma \ref{mixed types}, the only possible orbit partition containing a type I orbit is $\rm [I^2,III^4]$. However, we have seen that no pair of type III orbits can share an edge of valency 2 (\cf proof of Lemma \ref{orbit type1}(c)), while type I orbits have no edges of that valency. This implies orbits of type III have unusable edges of valency 2 from which the result follows.  
\end{proof}

\begin{lemm}\label{survivor}
A $K$-orbit partition of $\,\Gamma_2$ must   be of the form
$\rm [II,IV^4,V]$. 
\end{lemm}

\begin{proof}
By Lemmas \ref{no III,IV} and \ref{no type I}, a $K$-orbit partition must contain at least one orbit of type II or V.    Suppose there exists a type II orbit in the partition. 
Then by Lemma \ref{orbit type1} there cannot be a second orbit of type II, nor can there be more than one orbit of type V. Thus  the partition is of the form $\rm [II, III^c, IV^d, V^e]$ where $c+d+e=5$ and $0\le e\le 1$. 
Also, as  in the proof of Lemma \ref{no III,IV},
$c$ must be even. 

Suppose first that $c=2$.  Then by  Lemma \ref{mixed types} we have $d\le 2$, whence 
the partition must be of the form 
$\rm [II,III^2,IV^2,V]$.  In the proof of Lemma \ref{mixed types}(d), we saw that an edge shared by an orbit of type III and one of type IV must have valency 1 or 2. As type III orbits have only one edge of each such valency and there are two type IV orbits, both edges must be shared with a type IV  orbit.  But type III orbits have only one edge of valency 2, hence these edges are now used up.  This is a contradiction because type V orbits have four edges of valency 2.  

For $c=4$, the argument is similar. An edge shared by two type III orbits must have valency $0,3$ or 4. As there are four type III orbits in this case, every such edge is used between orbits of this type.  In particular, the single edge of valency 3 is no longer available.  But a type II orbit has four edges of valency 3, so we again reach a contradiction. This proves $c=0$ when a type II orbit is in the partition.

It only remains to see what occurs if we first assume the partition contains a type V orbit. But in this case the four valency 2 edges of this orbit are shared by four other orbits, be they of type III or IV.  This means there must be an orbit of type II in the partition, a case we have already treated. We conclude that $\rm [II, IV^4,V]$ is the only possible form of a $K$-orbit partition of $\Gamma_2$ as claimed. 
\end{proof}

We depict the orbit partition of type $\rm [II, IV^4,V]$ in Figure \ref{fig:ep}. It too will be shown to not exist  in due course.

\begin{figure}
\resizebox{9cm}{7.5cm}{
\begin{tikzpicture}
\draw[thick] (0,0) circle (6mm); 
\draw [thick](3,0) circle (6mm);
\draw [thick](0,3) circle (6mm)  node at (0,3) {4};
\draw[thick] (3,3) circle (6mm)  node at (3,3) {4};
\draw[thick] (-2.5,1.5) circle (6mm)  node at (-2.5,1.5) {0};
\draw[thick] (5.5,1.5) circle (6mm) node at (5.5,1.5) {4};

\draw (.6,0)--(2.4,0);
\draw (.6,3)--(2.4,3);
\draw (0,.6)--(0,2.4);
\draw (3,.6)--(3,2.4);
\draw (.4,.4)--(2.6,2.6);
\draw (.4,2.6)--(2.6,.4);

\node at (0,0) {4};
\node at (1.5,.238) {1};
\node at (.75,1.2) {1};
\node at (-.2,1.5) {1}; 
\node at (-1.2,.9) {3};
\node at (5,.095) {2}; 

\node at (3,0) {4};
\node at (3.2,1.5) {1}; 
\node at (2.2,1.2) {1};
\node at (-2.1,.095) {3}; 
\node at (4.2,.9) {2};

\node at (1.5,3.225) {1};
\node at (1.5,5.05) {0};

\node at (-2.1,3.1) {3}; 
\node at (4.95,3.1) {2};

\node at (-1.2,2.15) {3};
\node at (4.2,2.15) {2};

\node at (-.82,-.45) {\Large IV};
\node at (3.85,-.45) {\Large IV};
\node at (-.82,3.38) {\Large IV};
\node at (3.89,3.4) {\Large IV};

\node at (-3.4,1.5) {\Large II};
\node at (6.42,1.5) {\Large V};

\draw (-2,1.1)--(-.55,.17);
\draw (-2,1.9)--(-.55,2.83);

\draw (5,1.1)--(3.55,.17);
\draw (5,1.9)--(3.55,2.83);

\draw [bend right=80] (2.8,3.6) to (-2.6,2.1); 
\draw [bend left=80] (2.8,-.6) to (-2.6,.9); 

\draw [bend right=80] (0,-.6) to (5.5,.9); 
\draw [bend left=80] (0,3.6) to (5.5,2.1); 

\draw (-2.6,2.1) to[bend left=90, looseness=1.35] (5.5,2.1);

\end{tikzpicture}}
\vspace{-7mm}
   \caption{The lone surviving $K$-orbit partition of $\Gamma_2$ (\cf Lemma \ref{survivor})}\label{fig:ep}
\end{figure}
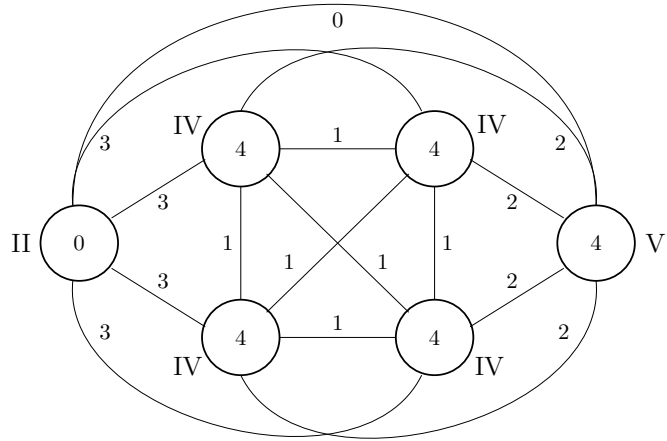

\begin{lemm}\label{no C3}
Neither a type {\rm IV} orbit nor a type {\rm V} orbit can contain a $3$-cycle.  
\end{lemm}		

\begin{proof}
First observe that type IV and type V orbits have internal valency 4.  In any orbit of either type we may fix a vertex $v$ and denote its neighbors by $v^p, v^{r}, v^s, v^{t}$ where $p,r,s,t\in K$. However $v^{\,p^{-1}},v^{\,r^{-1}}, v^{\,s^{-1}}, v^{\,t^{-1}}$ must  also be neighbors of $v$. This means, with one exception, the automorphisms $p,r,s,t$ must come in pairs. The one exception is if two or four of $p,r,s,t$ are involutory. However, this cannot be the case since  $K$ contains a unique involution.  Therefore, without loss of generality we may assume  $r=p^{-1}$ and 
   $t=s^{-1}$.  Moreover, we know that $|r|,|t|\in\{7, 14\}$.
    Three broad cases can arise here, namely $|r|=|t|=7$, $|r|=|t|=14$ and $|r|=7,\, |t|=14$.  
    
\smallskip    
    
 \noindent {\bf Case 1. $|r|=|t|=7$}: In this case the orbit is comprised of two connected components but that won't affect our argument.
   Since $K$ contains a unique subgroup of order 7, we must have $r=t^{\,m}$ for some integer $m\in \{2,3\}$. However, both subcases violate $\lambda=1$ as depicted in Figure \ref{|r|=|t|=7}.  Specifically, if $r=t^2$ then the edge $v^t v^{\,t^{\,2}}$ lies on two 3-cycles, while if $r=t^3$ the edge $v v^{\,t^{\,4}}$ suffers the same fate. (Note that the subcases $m=4,5$ are identical to $m=3,2$ respectively.) 
 
 \smallskip 
 
 \noindent {\bf Case 2. $|r|=|t|=14$}:  Here we have $r=t^{\,m}$ where $m\in \{3,5\}$. Both subcases violate $\mu=2$ as indicated in Figure \ref{|r|=|t|=14}. Specifically, If $r=t^{\,3}$ then vertices $v^{\,t^{\,3}}$,  $v^{\,t^{\,5}}$ have $v^{\,t^{\,2}},  v^{\,t^{\,4}}, v^{\,t^{\,6}}$ as common neighbors while 
if $r=t^{\,5}$ then vertices $v$, $v^{\,t^{\,4}}$ have $v^{\,t^{\,5}}$, $v^{\,t^{\,9}}$, $v^{\,t^{-1}}$ as common neighbors.  
 
 \smallskip
 
 \noindent {\bf Case 3. $|r|=7, |t|=14$}:  Here there are three subcases to consider, namely $r\in\{t^2,t^4,t^6\}$ as indicated in Figure \ref{|r|=7, |t|=14}. The first and last of these subcases lead to violations. Specifically, $r=t^{\,2}$ violates $\lambda=1$ since  the edge $vv^{\,t}$ lies on two 3-cycles with respective antipodal vertices $v^{\,t^{\,2}}$ and 
 $v^{\,t^{\,-1}}$. In contrast, $r=t^{\,6}$ violates $\mu=2$ since the vertices $v^{\,t^{\,4}}$ and 
 $v^{\,t^{\,11}}$ have $v^{\,t^{\,3}},v^{\,t^{\,10}},v^{\,t^{\,12}}$ as common neighbors (as well as $v^{\,t^{\,5}}$).  
 Curiously, the case $r=t^{\,4}$ does not lead to any $\lambda$ or $\mu$ violations, however it produces  no 3-cycles either.  This completes the proof of the lemma.
\end{proof}

 \begin{center}  
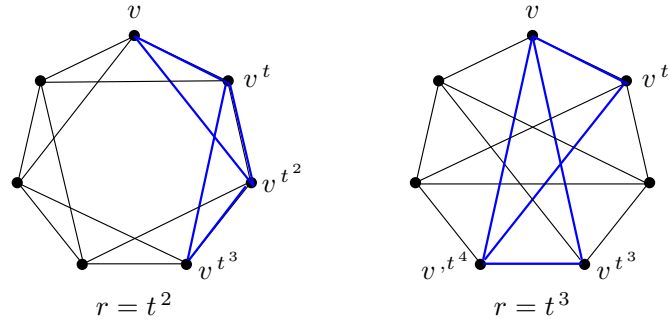
\begin{figure}
\resizebox{9cm}{!}{
\begin{tikzpicture}
 
 \node[draw, regular polygon, regular polygon sides=7, minimum size=3cm] (a) {};
\foreach \i in {1,2,...,7}
\fill (a.corner \i) circle[radius=2pt];
\draw[thick,blue] (0,1.5)--(1.5,-.38);
\draw[thick,blue] (1.15,.9)--(.67,-1.3); 

\draw[thick,blue] (0,1.5)--(1.2,.9);
\draw[thick,blue] (1.2,.9)--(1.5,-.38);

\draw[thick,blue] (1.4,-.38)--(.69,-1.3);

\node at (0,-1.85) {$r=t^{\,2}$};
\node at (0,1.8) {$v$};
\node at (1.55,1) {$v^{\,t}$};
\node at (1.87,-.276) {$v^{\,t^{\,2}}$};
\node at (1.07,-1.37) {$v^{\,t^{\,3}}$};

\begin{scope}[rotate=51.4+51.4]
\draw (1.15,.9)--(.67,-1.3); 
\end{scope}
 
\begin{scope}[rotate=51.4+51.4+51.4]
\draw (1.15,.9)--(.67,-1.3); 
\end{scope}
 
\begin{scope}[rotate=51.4+51.4+51.4+51.4]
\draw (1.15,.9)--(.67,-1.3); 
\end{scope}

\begin{scope}[rotate=51.4+51.4+51.4+51.4+51.4]
\draw (1.15,.9)--(.67,-1.3); 
\end{scope}
 
\begin{scope}[rotate=51.4+51.4+51.4+51.4+51.4+51.4]
\draw (1.15,.9)--(.67,-1.3); 
\end{scope}

 \begin{scope}[xshift=5cm]
\node[draw, regular polygon, regular polygon sides=7, minimum size=3cm] (a) {};
\foreach \i in {1,2,...,7}
   \fill (a.corner \i) circle[radius=2pt];
   \draw[thick,blue] (0,1.5)--(.62,-1.3);
     \draw [thick, blue](1.15,.9)--(-.6,-1.3);
     \draw (1.5,-.351)--(-1.5,-.351); 
       \draw (-1.15,.9)--(.6,-1.3);
        \draw[thick,blue] (0,1.5)--(-.62,-1.3);
        \draw[thick,blue] (0,1.5)--(1.2,.9);
\draw[thick,blue] (.62,-1.35)--(-.62,-1.35);
  \node at (0,-1.855) {$r=t^{\,3}$};
\node at (0,1.8) {$v$};
 \node at (1.55,1) {$v^{\,t}$};
 \node at (1.07,-1.37) {$v^{\,t^{\,3}}$};
 \node at (-1.09,-1.37) {$v^{\,,t^{\,4}}$};

\begin{scope}[rotate=51.4]
\draw  (0,1.5)--(.62,-1.3); 
\end{scope}

\begin{scope}[rotate=51.4+51.4]
\draw  (0,1.5)--(.62,-1.3); 
\end{scope}

 \end{scope}
 
\end{tikzpicture}}
\caption{The case $|r|=|t|=7$ of Lemma \ref{no C3}}\label{|r|=|t|=7}
\end{figure}  
\end{center}  

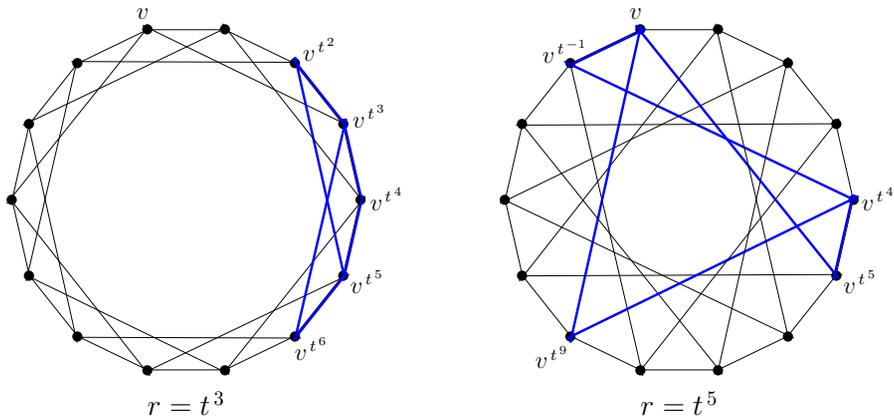
\begin{figure}
\resizebox{12cm}{!}{
\begin{tikzpicture}
 
 \node[draw, regular polygon, regular polygon sides=14, minimum size=6cm] (a) {};
\foreach \i in {1,2,...,14}
\fill (a.corner \i) circle[radius=2.5pt];
\draw (-.75,3)--(2.75,1.3);
\node at (-.75,3.2) {\Large$v$};
\node at (0,-3.5) {\LARGE$r=t^{\,3}$};
 
\begin{scope}[rotate=25.714]
\draw (-.75,3)--(2.75,1.3); 
\end{scope}

\begin{scope}[rotate=51.428]
\draw (-.75,3)--(2.75,1.3);
\end{scope}

\begin{scope}[rotate=77.142]
\draw (-.75,3)--(2.75,1.3);
\end{scope}

\begin{scope}[rotate=102.856]
\draw (-.75,3)--(2.75,1.3);
\end{scope}

\begin{scope}[rotate=102.856+25.714]
\draw (-.75,3)--(2.75,1.3);
\end{scope}

\begin{scope}[rotate=102.856+51.428]
\draw (-.75,3)--(2.75,1.3);
\end{scope}

\begin{scope}[rotate=102.856+77.142]
\draw (-.75,3)--(2.75,1.3);
\end{scope}

\begin{scope}[rotate=102.856+102.856]
\node at (-.75,3.2) {\Large$\;\,v^{\,t^{\,6}}$};
\draw (-.75,3)--(2.75,1.3);
\end{scope}

\begin{scope}[rotate=102.856+102.856+25.714]
\node at (-.75,3.2) {\Large$\;\;\,v^{\,t^{\,5}}$};
\draw (-.75,3)--(2.75,1.3);
\draw [very thick,blue](-.75,2.95)--(.75,2.95);
\end{scope}

\begin{scope}[rotate=102.856+102.856+51.428]
\draw (-.75,3)--(2.75,1.3); \node at (-.75,3.2) {\Large$\;\;\; v^{\,t^{\,4}}$};
\draw [very thick,blue](-.75,2.95)--(.75,2.95);
\end{scope}

\begin{scope}[rotate=102.856+102.856+77.142]
\draw [very thick,blue](-.75,3)--(2.75,1.3);  \node at (-.75,3.2) {\Large$\quad v^{\,t^{\,3}}$};
\draw [very thick, blue](-.75,2.95)--(.75,2.95); 
\end{scope}

\begin{scope}[rotate=102.856+102.856+102.856]   
\node at (-.75,3.2) {\Large$\quad\; v^{\,t^{\,2}}$};
\draw [very thick, blue](-.75,3)--(2.75,1.3);
\draw [very thick,blue](-.75,2.95)--(.75,2.95);
\end{scope}

\begin{scope}[rotate=102.856+102.856+102.856+25.714]
\draw (-.75,3)--(2.75,1.3);
\end{scope}

\begin{scope}[xshift=8.5cm]

 \node[draw, regular polygon, regular polygon sides=14, minimum size=6cm] (a) {};
\foreach \i in {1,2,...,14}
\fill (a.corner \i) circle[radius=2.5pt];
\draw[very thick,blue] (-.75,3)--(2.75,-1.35);
\node at (-.75,3.2) {\Large$v$};
\node at (0,-3.5) {\LARGE\bf$r=t^{\,5}$};
 
\begin{scope}[rotate=25.714]
\node at (-.75,3.2) {\Large$\;\,v^{\,t^{\,-1}}$};
\draw[very thick,blue] (-.75,3)--(2.75,-1.35);
\draw [very thick,blue] (-.75,2.9)--(.75,2.9);
\end{scope}

\begin{scope}[rotate=51.428]
\draw (-.75,3)--(2.75,-1.35);
\end{scope}

\begin{scope}[rotate=77.142]
\draw (-.75,3)--(2.75,-1.35);
\end{scope}

\begin{scope}[rotate=102.856]
\draw (-.75,3)--(2.75,-1.35);
\end{scope}

\begin{scope}[rotate=102.856+25.714]
\node at (-.75,3.4) {\Large$v^{\,t^{\,9}}$};
\draw  [very thick,blue] (-.75,3)--(2.75,-1.35);
\end{scope}

\begin{scope}[rotate=102.856+51.428]
\draw (-.75,3)--(2.75,-1.35);
\end{scope}

\begin{scope}[rotate=102.856+77.142]
\draw(-.75,3)--(2.75,-1.35);
\end{scope}

\begin{scope}[rotate=102.856+102.856]
\draw (-.75,3)--(2.75,-1.35);
\end{scope}

\begin{scope}[rotate=102.856+102.856+25.714]
\draw (-.75,3)--(2.75,-1.35);
\node at (-.75,3.2) {\Large$\;\,\;v^{\,t^{\,5}}$};
\end{scope}

\begin{scope}[rotate=102.856+102.856+51.428]
\draw [very thick,blue](-.75,3)--(2.75,-1.35); 
\node at (-.75,3.2) {\Large$\;\,\;v^{\,t^{\,4}}$};
\draw [very thick,blue] (-.75,2.9)--(.75,2.9);
\end{scope}

\begin{scope}[rotate=102.856+102.856+77.142]
\draw (-.75,3)--(2.75,-1.35);
\end{scope}

\begin{scope}[rotate=102.856+102.856+102.856]
\draw (-.75,3)--(2.75,-1.35);
\end{scope}

\begin{scope}[rotate=102.856+102.856+102.856+25.714]
\draw (-.75,3)--(2.75,-1.35);
\end{scope}

\end{scope}

\end{tikzpicture}}
 
\caption{The case $|r|=|t|=14$ of Lemma \ref{no C3}}\label{|r|=|t|=14}
\end{figure}

\begin{figure}
\resizebox{12cm}{!}{
\begin{tikzpicture}
 
 \node[draw, regular polygon, regular polygon sides=14, minimum size=6cm] (a) {};
\foreach \i in {1,2,...,14}
   \fill (a.corner \i) circle[radius=2.5pt];
\draw[very thick,blue] [bend right=20] (-.75,3) to (2,2.3);
\node at (-.75,3.2) {\Large$v$};
\node at (0,-3.5) {\LARGE$r=t^{\,2}$};
\draw [very thick,blue] (-.6,2.9)--(.6,2.93);
 
\begin{scope}[rotate=25.714]
\draw [very thick,blue] [bend right=20] (-.75,3) to (2,2.3); 
\node at (-.75,3.2) {\Large$v^{\,t^{\,-1}}$};
\draw [very thick,blue] (-.6,2.9)--(.6,2.93);
\end{scope}

\begin{scope}[rotate=51.428]
\draw [bend right=20] (-.75,3) to (2,2.3);
\end{scope}

\begin{scope}[rotate=77.142]
\draw [bend right=20] (-.75,3) to (2,2.3);
\end{scope}

\begin{scope}[rotate=102.856]
\draw [bend right=20] (-.75,3) to (2,2.3);
\end{scope}

\begin{scope}[rotate=102.856+25.714]
\draw [bend right=20] (-.75,3) to (2,2.3);
\end{scope}

\begin{scope}[rotate=102.856+51.428]
\draw [bend right=20] (-.75,3) to (2,2.3);
\end{scope}

\begin{scope}[rotate=102.856+77.142]
\draw [bend right=20] (-.75,3) to (2,2.3);
\end{scope}

\begin{scope}[rotate=102.856+102.856]
\draw [bend right=20] (-.75,3) to (2,2.3);
\end{scope}

\begin{scope}[rotate=102.856+102.856+25.714]
\node[draw, regular polygon, regular polygon sides=14, minimum size=6cm] (a) {};
\draw [bend right=20] (-.75,3) to (2,2.3);
\end{scope}

\begin{scope}[rotate=102.856+102.856+51.428]
\node[draw, regular polygon, regular polygon sides=14, minimum size=6cm] (a) {};
\end{scope}

\begin{scope}[rotate=102.856+102.856+77.142]
\draw [bend right=20] (-.75,3) to (2,2.3);
\end{scope}

\begin{scope}[rotate=102.856+102.856+102.856]
\draw [bend right=10] (-.75,3) to (2,2.3);
\node at (-.75,3.2) {\Large\quad\;$v^{\,t^{\,2}}$};
\end{scope}

\begin{scope}[rotate=102.856+102.856+102.856+25.714]
\draw [bend right=20] (-.75,3) to (2,2.3);
\node at (-.75,3.2) {\Large\;\;\;$v^{\,t}$};
\draw [very thick,blue] (-.6,2.9)--(.6,2.93);
\end{scope}

\begin{scope}[xshift=8.5cm]

\node[draw, regular polygon, regular polygon sides=14, minimum size=6cm] (a) {};
\foreach \i in {1,2,...,14}
\fill (a.corner \i) circle[radius=2.5pt];
\draw (-.75,3)--(1.9,-2.4);
\node at (-.75,3.2) {\Large$v$};
\node at (0,-3.5) {\LARGE$r=t^{\,6}$};
 
\begin{scope}[rotate=25.714]
\draw (-.75,3)--(1.9,-2.4);
\end{scope}

\begin{scope}[rotate=51.428]
\node at (-.75,3.3) {\Large$\;\,v^{\,t^{\,12}}$};
\draw [very thick,blue] (-.75,3)--(1.9,-2.4);
\end{scope}

\begin{scope}[rotate=77.142]
\node at (-.75,3.4) {\Large$\;\,v^{\,t^{\,11}}$};
\draw [very thick,blue] (-.75,3)--(1.9,-2.4);
\draw [very thick,blue] (-.6,2.9)--(.6,2.93);
\end{scope}

\begin{scope}[rotate=102.856]
\node at (-.75,3.5) {\Large$\;\,v^{\,t^{\,10}}$};
\draw (-.75,3)--(1.9,-2.4);
\draw [very thick,blue] (-.6,2.9)--(.6,2.93);
\end{scope}

\begin{scope}[rotate=102.856+25.714]
\draw (-.75,3)--(1.9,-2.4);
\end{scope}

\begin{scope}[rotate=102.856+51.428]
\draw (-.75,3)--(1.9,-2.4);
\end{scope}

\begin{scope}[rotate=102.856+77.142]
\draw (-.75,3)--(1.9,-2.4);
\end{scope}

\begin{scope}[rotate=102.856+102.856]
\draw (-.75,3)--(1.9,-2.4);
\end{scope}

\begin{scope}[rotate=102.856+102.856+25.714]
\draw (-.75,3)--(1.9,-2.4);
\end{scope}

\begin{scope}[rotate=102.856+102.856+51.428]
\draw  [very thick,blue](-.75,3)--(1.9,-2.4);
\node at (-.75,3.2) {\Large$\;\,\;v^{\,t^{\,4}}$};
\end{scope}

\begin{scope}[rotate=102.856+102.856+77.142]
\node at (-.75,3.2) {\Large$\;\,\;v^{\,t^{\,3}}$};
\draw (-.75,3)--(1.9,-2.4);\draw [very thick,blue] (-.6,2.9)--(.6,2.93);

\end{scope}

\begin{scope}[rotate=102.856+102.856+102.856]
\draw (-.75,3)--(1.9,-2.4);
\end{scope}

\begin{scope}[rotate=102.856+102.856+102.856+25.714]
\draw (-.75,3)--(1.9,-2.4);
\end{scope}

\end{scope}
 
\end{tikzpicture}}

\resizebox{5.5cm}{!}{
\begin{tikzpicture}
 
 \node[draw, regular polygon, regular polygon sides=14, minimum size=6cm] (a) {};
\foreach \i in {1,2,...,14}
\fill (a.corner \i) circle[radius=2.5pt];
\draw (-.75,3)--(3,-.035);\draw (.62,3)--(2.67,-1.25);
\node at (-.75,3.2) {\Large$v$};
\node at (0,-3.5) {\LARGE$r=t^{\,4}$};
 
\begin{scope}[rotate=51.428]
\draw (-.75,3)--(3,-.035);\draw (.62,3)--(2.67,-1.25);
\end{scope}

\begin{scope}[rotate=102.856]
\draw (-.75,3)--(3,-.035);\draw (.62,3)--(2.67,-1.25);
\end{scope}

\begin{scope}[rotate=102.856+51.428]
\draw (-.75,3)--(3,-.035);\draw (.62,3)--(2.67,-1.25);
\draw (-.6,2.9)--(.6,2.93);

\end{scope}

\begin{scope}[rotate=102.856+102.856]
\draw (-.75,3)--(3,-.035);\draw (.62,3)--(2.67,-1.25);
\end{scope}

\begin{scope}[rotate=102.856+102.856+51.428]
\draw (-.75,3)--(3,-.035);\draw (.62,3)--(2.67,-1.25);\node at (-.74,3.37) {\Large$v^{\,t^{\,4}}$};
\end{scope}

\begin{scope}[rotate=102.856+102.856+102.856]
\draw (-.75,3)--(3,-.035);\draw (.62,3)--(2.67,-1.25);
\end{scope}

\end{tikzpicture}}
\caption{The case $|r|=7, |t|=14$ of Lemma \ref{no C3}}\label{|r|=7, |t|=14}
\end{figure}
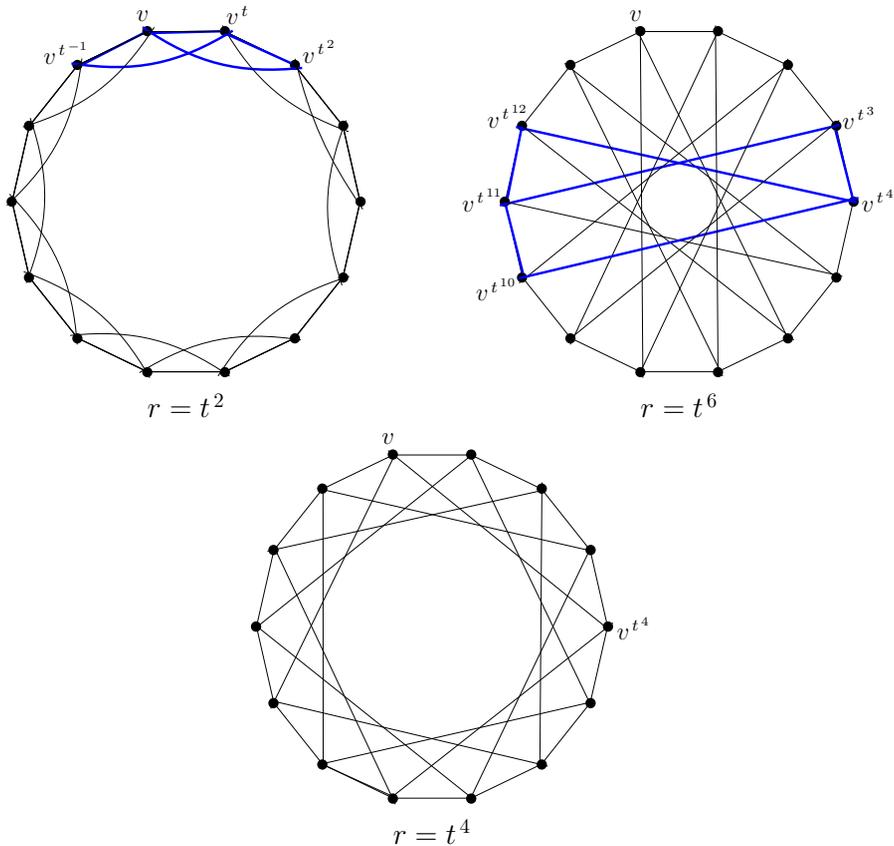   

\begin{thm}\label{no 14}
$G$ does not contain any order $14$ elements.  
\end{thm}

\begin{proof}
By way of contradiction, suppose $G$ contains an element of order $14$. By Lemma 
\ref{order 14 element} we may assume this element is $st$, and as usual we set $K=\langle st\rangle$.  
Lemma \ref{survivor} now asserts that a $K$-orbit partition must be of the form 
$\rm [II,IV^4,V]$. (See Figure \ref{fig:ep}.)  
We now show this cannot occur.

Let $\mathcal O$ be one of the four type IV orbits in the partition. Clearly $\mathcal O$ has 28 internal edges each of which must lie on a unique 3-cycle.  By Lemma 
\ref{no C3}, the third vertex of each such 3-cycle must lie in an orbit other than $\mathcal O$.  However, this vertex cannot lie in any other type IV orbit since every pair of type IV orbits are conjoined by an orbit edge of valency 1.  Thus the 28 vertices used to complete the aforementioned  3-cycles must lie in the two remaining orbits of type II and V.  Since there are a total of 28 vertices in these two orbits, we conclude that every one of these vertices gets used in this fashion. Of great importance to us is the fact that every vertex in the orbit of type V lies on a unique 3-cycle with an edge in $\mathcal O$.  

The above argument applies equally well to each type IV orbit.  This means that each vertex in the unique type V orbit $\,\mathcal O'$ must lie on exactly four 3-cycles conceived in this manner.  However, $\,\mathcal O'$ has no internal 3-cycles and there are no edges between $\,\mathcal O'$ and the type II orbit of the partition. Thus every vertex in $\,\mathcal O'$ lies on a total of four 3-cycles in $\Gamma_2$,  which contradicts the fact that every vertex of   
$\Gamma_2$ is required to lie on five 3-cycles in $\Gamma_2$ (\cf  Lemma \ref{140 triangles}).  We conclude that $G$ cannot contain any element of order 14.  
\end{proof}

\section{Consequences of divisibility by 7}\label{cyclic or Frobenius} 
 
\begin{prop}\label{2 divides}
$|G|$ is not divisible by $14$. 
\end{prop}

\begin{proof}
Suppose 14 divides $|G|$.  Then by \cite{Makhnev} one has $|G|\in \{14,42\}$. Recall from  \cite{Makhnev2} that  $[O(G),t]=1$, where $O(G)$ is the maximal odd order normal subgroup of $G$ and $t\in G$ is an involution.  This alone rules out $D_{14}$, $D_{42}$,  $D_{14}\times \mathbb Z_3$ and   $Frob(42)$ as possible isomorphism types of $G$. The only remaining possibilities  are 
$\mathbb Z_{14}$, $\mathbb Z_{42}$,  ${\mathbb Z_7}\times S_3$ and 
${Frob}(21)\times \mathbb Z_2$.  However, each of these groups has an element of order 14 so is ruled out by Theorem \ref{no 14}.
\end{proof}

\begin{lemm}\label{P7 normal}
Suppose $7$ divides $|G|$, and let $P_7$ be a Sylow $7$-subgroup of $G$.  Then $P_7$ is normal in $G$.  
\end{lemm}

\begin{proof}
By  \cite{Makhnev} and Proposition \ref{2 divides}, $|G|$ must divide $3^3\cdot 7\cdot 11$. By Sylow's Theorem, the number $n_7$ of Sylow $7$-subgroups must satisfy $n_7=[G\!:\!N_G(P_7)]$ and $n_7\equiv 1\!\pmod{7}$. It is straightforward to deduce that $P_7\trianglelefteq G$ for all orders of $G$ except possibly $3^2\cdot 7\cdot 11$ and $3^3\cdot 7\cdot 11$.
However, in these two cases  one has $P_{11}\trianglelefteq G$ where $P_{11}$ is a Sylow 11-subgroup of $G$. 
As $P_7$ does not embed in ${\rm Aut}(P_{11})\cong \mathbb Z_{10}$ it follows that $[P_7,P_{11}]=1$. But then $P_{11}\le N_G(P_7)$ whence $n_7=[G\!:\!N_G(P_7)]\in\{1,9,27\}$.  We now conclude from the congruence $n_7\equiv 1\!\pmod{7}$ that
 $n_7=1$,  \ie $P_7\trianglelefteq G$ as claimed.
\end{proof}

 \begin{prop}\label{if 7 divides}
 Divisibility  by $7$ implies $|G|$ divides $21$.
 \end{prop}

 \begin{proof}
  By our labeling scheme in Section \ref{prelims} and Remark  \ref{|s|=7}, it is clear that 
 $P_7=\langle s\rangle$   where $s=(1L,2L,\dots ,7L)\,(1R,2R,\dots ,7R)$. Furthermore, as $\langle s\rangle\trianglelefteq G$ by Lemma \ref{P7 normal}, we have that $G$ embeds in the normalizer $N_S(\langle s\rangle)$ where 
 $S={\rm Sym}(V(\Gamma_1))\cong S_{14}$.  Thus the order of $G$ must simultaneously divide   
 $|N_S(\langle s\rangle)|=2^2\cdot 3\cdot 7^2$ and $3^3\cdot 7\cdot 11$. Obviously this implies $|G|$ divides $21$.   
 \end{proof}
 
 \begin{coro}\label{two groups}
 \hspace*{-6mm}

\begin{enumerate}
 
\item[\rm (a)]  
If $\,2$ divides $|G|$, then $|G|$ divides $6$.
\item[\rm (b)] 
If $7$ divides $|G|$,  then $G$  is isomorphic to either $\,\mathbb Z_7$ or $\,Frob(21)$.

\end{enumerate}
\end{coro}

\begin{proof}
As divisibility by 2 implies $|G|$ divides 42, part (a) follows directly from Proposition \ref{2  divides}.  To prove (b),  
observe  by
Proposition \ref{if 7 divides} that $|G|\in\{7,21\}$. There are only two groups of order 21 up to isomorphism, namely $\mathbb Z_{21}$ and $Frob(21)$, thus we have only to rule out the former.  But 3 does not divide the order of the centralizer  $C_G(s)=(\langle s_L\rangle\times \langle s_R\rangle)\rtimes\langle t
\rangle\cong (\mathbb Z_7\times \mathbb Z_7)\rtimes \mathbb Z_2$ where   $S={\rm Sym}(V(\Gamma_1))$. Thus $G$ cannot contain any  element of order 21, \ie $G\not\cong\mathbb Z_{21}$.  
\end{proof}

In what follows, we gather detailed information about a putative Conway 99-graph $\Gamma$ under the assumption that 
$G\cong Frob(21)$.  
 Throughout, we assume $G=\langle s,r\rangle$ with $|r|=3$.    Since  $r$ normalizes $\langle s\rangle$ it is clear that $r$ fixes the root vertex $x$, hence $\Gamma_1$ and $\Gamma_2$ are preserved under the action of $G$.  Still, we have yet to  pin down the precise structure of $r$.  This is  remedied below. 

\begin{lemm}
With notation as above, we may assume 
\[r=(1L, 2L, 4L)\,(3L, 6L, 5L)\,(1R, 2R, 4R)\,(3R, 6R, 5R).\] 
\end{lemm}

\begin{proof}
It is immediate that $r$ is an element in the normalizer $N_S(\langle s\rangle)$ in 
$S={\rm Sym}(V(\Gamma_1))$.  Applying standard group theoretic arguments, one deduces that there are 98  elements of order 3 in $N_S(\langle s\rangle)$ and these take the form 
\[\left(\big[(1L, 2L, 4L)\,(3L, 6L, 5L)\big]^{s_L^{\,i}} \,\big[(1R, 2R, 4R)\,(3R, 6R, 5R)\big]^{s_R^{\,j}}\right)^{\pm 1}\]
where $0\le i,j\le 6$. However, it is easily verified that such an element is an  automorphism of $\Gamma_1$  if and only if $i=j$. Thus, the 14 elements of order 3 in $G$ are precisely 
\[\big((1L, 2L, 4L)\,(3L, 6L, 5L)\,(1R, 2R, 4R)\,(3R, 6R, 5R)\big)^{\pm\, s^{\,i}}\]
from which the result follows.
\end{proof}

\begin{lemm}\label{1, 7^2, 21^4}
The $G$-orbit structure on $\Gamma$ is $[1, 7^2, 21^4]$. 
\end{lemm}

\begin{proof}
As  21 does not divide 14, the two $\langle s\rangle$-orbits on $\Gamma_1$ cannot fuse under the action of $G$.  Thus the $G$-orbit structure on $\Gamma_1$ is $[7^2]$.

Now let $\mathcal O$ be an arbitrary $G$-orbit on $\Gamma_2$. 
We claim $G$ acts regularly on $\mathcal O$ from which the desired result will follow.  To this end, suppose $ijXY \in \mathcal O$ is fixed by some $g\in G$, 
 where $X, Y \in \{L, R\}$.  It follows that $|g|=3$, since $s$ has no fixed points in $\Gamma_2$.  As $\lambda =1$, we have $i\ne j$. 
Thus since $\mu=2$, $g$ must either fix or interchange the two 
$\Gamma_1$-neighbors of $ijXY$, \ie  
$\{iX^g,jY^g\}=\{iX,jY\}$. But as $G$ has odd order, these vertices must be fixed by $g$, that is, $iX^g=iX$ and $jY^g=jY$.  As $g$ preserves adjacency, we must also have 
$(iX^C)^g=iX^C$ and  $(jY^C)^g=jY^C$ where $\{X,X^C\}=\{Y,Y^C\}=\{L,R\}$.
In every instance, we obtain $iL^g=iL$ and $jL^g=jL$.  By  transitivity of  $\langle s\rangle$ on the orbit containing $iL$, we get $iL=jL^z$ for some $z\in \langle s\rangle$.  This gives $jL^{[z,g]}=iL^{g^{\,-1}zg}=iL^{zg}=jL^{g}=jL$, \ie $jL$ is a fixed point of $[z,g]\in \langle s\rangle$. But this implies $[z,g]=1$, a contradiction since $G$ is nonabelian.
We conclude that the $G$-orbit structure on $\Gamma_2$ is $[21^4]$ as claimed.
\end{proof}

\begin{rema}
Now that we understand the orbit structure of $G$ on $\Gamma_2$, it is an easy matter to determine a corresponding set of orbit representatives, viz.\ $\{12LL, 12RR, 12LR, 12RL\}$.  For brevity we shall denote these orbits as $LL, RR, LR, RL$, respectively. We indicate the corresponding $G$-orbit partition of $\,\Gamma_2$ in Figure \ref{general OP21}.
\end{rema}

\begin{figure}
\begin{tikzpicture}
\draw[thick] (0,0) circle (6mm) node at (0,0) {\Large$b_1$}  node at (0,-.9) {$LL$};
\draw [thick](3,0) circle (6mm) node at (3,0) {\Large$b_2$}  node at (3,-.9) {$RR$};
\draw [thick](0,3) circle (6mm) node at (0,3) {\Large$b_3$}  node at (0,3.9) {$LR$};
\draw[thick] (3,3) circle (6mm) node at (3,3) {\Large$b_4$}  node at (3,3.9) {$RL$};
 
\draw (.6,0)--(2.4,0) node at (1.54,-.27) {\large$b_{12}$};
\draw (.6,3)--(2.4,3) node at (-.33,1.5) {\large$b_{13}$};
\draw (0,.6)--(0,2.4) node at (3.36,1.5) {\large$b_{24}$};
\draw (3,.6)--(3,2.4) node at (1.5,3.27) {\large$b_{34}$};
\draw (.4,.4)--(2.6,2.6) node at (.64,1.1) {\large$b_{14}$};
\draw (.4,2.6)--(2.6,.4) node at (2.34,1.1) {\large$b_{23}$};

\end{tikzpicture}
\caption{General form of an $G$-orbit partition of $\Gamma_2$}\label{general OP21}
\end{figure}
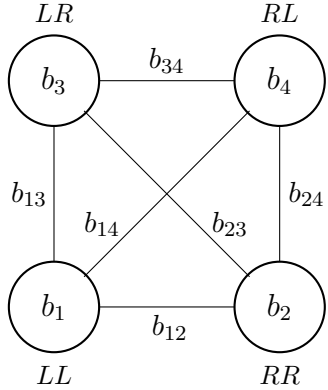

For $ijXY\in V(\Gamma_2)$, we refer to $iX$ and $jY$   as its  coordinates. In the result that follows, we demonstrate a  manner in which the 
 coordinates of the 12 $\Gamma_2$-neighbors of a fixed vertex in $\Gamma_2$ are balanced.  

\begin{lemm}\label{balance} Let $ijXY$ be a fixed but arbitrary vertex in $\Gamma_2$  and consider collectively the $24$ coordinates appearing among its $12$ $\Gamma_2$-neighbors. Then the following hold: 
\begin{enumerate}
\item Each of $iL,iR,jL,jR$ appears exactly once. 
\item Each of $kL, kR$ appears exactly twice for each $k\ne i,j$. 
\item Each of $L$ and $R$ appears $12$ times, \ie $\Gamma_2$-neighbors of a vertex in  $\Gamma_2$  are $L/R$-balanced.
\end{enumerate} 
\end{lemm}

\begin{proof}
As in Lemmas \ref{orbits} and \ref{1, 7^2, 21^4} we let 
$\{X,X^C\}=\{Y,Y^C\}=\{L,R\}$. We must show that
each of $iX,jY$ and $iX^C,jY^C$ occur exactly once as coordinates of $\Gamma_2$-neighbors of $iXjY$.  We treat the pair $iX,jY$ first. 
Since $\lambda=1$ and $ijXY$ is adjacent to $iX\in V(\Gamma_1)$, there is a unique vertex in $\Gamma_2$ which is their common neighbor. Evidently this vertex is of the form $i\ell XW$ for some $W\in\{L,R\}$ with $\ell W\ne jY$. Thus $iX$ occurs exactly once as a coordinate of a $\Gamma_2$-neighbor of $ijXY$ and by a symmetric argument the same result holds for $jY$.  

We next treat the pair $iX^C,jY^C$. 
Observe that since $\lambda=1$, we have that $ijXY$ and $iX^C$ are nonadjacent.  Since $\mu=2$, the vertices $ijXY$ and $iX^C$ must have a unique common neighbor $i\ell X^CW\in V(\Gamma_2)$ with $iX$ being their second common neighbor. Thus $iX^C$ occurs exactly once as a  coordinate of a $\Gamma_2$-neighbor of $ijXY$ with a similar result holding for $jY^C$.  As $\{iX,iX^C\}=\{iL,iR\}$ and $\{jY,jY^C\}=\{jL,jR\}$, assertion (1) is proved. 

Now let $k\ne i,j$ and $W\in\{L,R\}$. Since $\mu=2$ and $kW$  is nonadjacent to 
$ijXY$,  there are exactly two vertices in $\Gamma_2$ that 
are common neighbors of $kW$ and $ijXY$. Obviously, $kW$ appears as a coordinate in each such neighbor, so twice in total. Thus assertion (2) is proved. 
Finally, observe that (3) follows directly from (1) and (2).   
\end{proof}

As a consequence of Lemma \ref{balance}(3), we have the following.

\begin{lemm}
 With notation as in Figure \ref{general OP21}, $b_1=b_{12}=b_2$, 
$\,b_{13}=b_{23}\,$ and $\,b_{14}=b_{24}$.\label{more relations} 
\end{lemm}

\begin{proof}
We  apply Lemma \ref{balance}(3) to vertices in $LL$, $RR$, $LR$, $RL$ in that order. Prefatory to this, note that the coordinates of vertices in $LR$ and $RL$ have a natural $L/R$-balance built into them. This means we may safely ignore edges that adjoin any vertex of $\Gamma_2$ to vertices in either of  these two orbits. 

Let $u$ be a fixed vertex in $LL$.  As $u$ has valency $b_1$ in $LL$, it must have $b_1$ neighbors in $RR$ in order to restore $L/R$-balance.  This proves 
$b_{12}=b_1$. By a similar argument based on $RR$, we have  $b_{12}=b_2$.  
Now let $u$ be a vertex in $LR$. Clearly, every neighbor of $u$ in $LL$ must be reciprocated by a neighbor in $RR$ to maintain balance. This proves $b_{13} = b_{23}$. Applying this  argument to vertices in $RL$ now yields
$b_{14}=b_{24}$. \end{proof}

To further narrow down possible orbit valencies, we adopt the approach  used in Section \ref{no order 14 auts}. That is to say, we analyze
 2-paths between vertices from pairs of orbits.  Let $u$ be a fixed vertex in the orbit $LL$.  We wish to count in  
two ways the cardinality of the following set:
\[S = \{uvw: \text{$uvw$ is a 2-path with } w \in RR\}.\]  
Observe that for each of the $b_1$ neighbors of $u$ in $RR$ there is a unique 2-path in $S$ (since $\lambda =1$), while for each of the $21-b_1$ non-neighbors of $u$ in $RR$ there are two such paths (since $\mu=2$).   
This gives $|S|=b_1\cdot 1 + (21-b_1)\cdot 2=42-b_1$. 

We next condition our count on the location of the intermediate vertex $v$.  Here we rely  heavily on Lemma \ref{more relations}.  

If $v$ lies in $LL$, there are $b_1$ choices for $v$ followed by 
$b_1$ independent choices for $w$.  This gives a total of $b_1^2$ paths in $S$ assuming the intermediate vertex $v$ is in $LL$.  Note that for $v$ in $RR$ the count is identical, \ie there are $b_1^2$ paths in $S$ assuming  $v$ is in $RR$.

Now suppose $v$ is in the orbit $LR$.  Since $u$ has $b_{13}$ neighbors in $LR$ and each such vertex has $b_{13}$ neighbors in $RR$, the total number of paths in this case is $b_{13}^2$.  By a similar argument the number of paths with $v$ in $RL$ is $b_{14}^2$.  Finally, we observe that there is no 2-path  with intermediate vertex $v$ in $V(\Gamma_1)$. Indeed, this would require that $v=iL$ and $v= jR$ for some $i,j$, which is absurd.  Thus we obtain $|S|=2b_1^2 + b_{13}^2+b_{14}^2$.

Equating the above two expressions for $|S|$ gives
$42-b_1= 2b_1^2 + b_{13}^2+b_{14}^2$ or equivalently  
\begin{equation} 42-b_1-2b_1^2 = b_{13}^2+b_{14}^2.\end{equation}
The only integral solution to (3) is $b_1=2$, $b_{13}=b_{14}=4$ which, by virtue of Lemma \ref{more relations}, narrows down all possible orbit valencies to the ones indicated in 
Figure \ref{updated relations}.

\begin{figure}
\begin{tikzpicture}
\draw[thick] (0,0) circle (6mm) node at (0,0) {\Large$2$}  node at (0,-.9) {$LL$};
\draw [thick](3,0) circle (6mm) node at (3,0) {\Large$2$}  node at (3,-.9) {$RR$};
\draw [thick](0,3) circle (6mm) node at (0,3) {\Large$b_3$}  node at (0,3.9) {$LR$};
\draw[thick] (3,3) circle (6mm) node at (3,3) {\Large$b_3$}  node at (3,3.9) {$RL$};
 
\draw (.6,0)--(2.4,0) node at (1.5,-.27) {\large$2$};
\draw (.6,3)--(2.4,3) node at (-.2,1.5) {\large$4$};
\draw (0,.6)--(0,2.4) node at (3.25,1.5) {\large$4$};
\draw (3,.6)--(3,2.4) node at (1.5,3.25) {\large$4- b_{3}$};
\draw (.4,.4)--(2.6,2.6) node at (.68,1.1) {\large$4$};
\draw (.4,2.6)--(2.6,.4) node at (2.32,1.1) {\large$4$};

\end{tikzpicture}
   \caption{Narrowing down the valencies of a $G$-orbit  partition of $\Gamma_2$}\label{updated relations}
\end{figure}
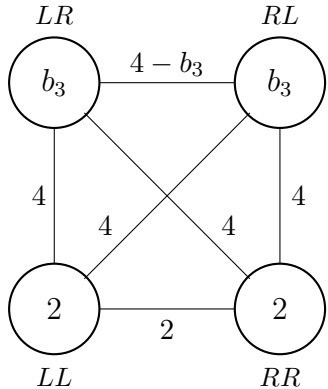

We are nearly at the point of determining the unique orbit structure of $G\cong Frob(21)$ acting on the second subconstituent $\Gamma_2$ of a putative Conway 99-graph $\Gamma$.   To complete the process, we count in two ways the number of 2-paths originating at a fixed vertex $u$ in $LR$ and ending at some vertex in $RL$.

One one hand, there are $4-b_{3}$ edges from $u$ to some vertex $w$ in $RL$, and in each case there is  a unique 2-path $uvw$ since $\lambda =1$. Similarly, for each of the $21-(4-b_{3})=17+b_3$ vertices in $RL$ nonadjacent to $v$ there are two distinct 2-paths of required form.  This gives a total of $(4-b_{3})\cdot1+ (17+b_{3})\cdot 2=38+b_{3}$ such 2-paths.  

We next focus our count on the location of the intermediate vertex $v$.  If $v$ is in either of $LL$ or $RR$, there are $4\cdot 4=16$
 such 2-paths, while if $v$ is in either of $LR$ or $RL$ there are 
 $(4-b_3)b_3 = 4b_3 - b_3^{\,2}$ such 2-paths. 
 Lastly, we consider $v\in V(\Gamma_1)$.  
Since $u$ is in $LR$, it has coordinates $ijLR$ for some $i\ne j$ whence its two 
$\Gamma_1$-neighbors are $iL$ and $jR$.  Note that the 12 neighbors of $iL$ in $\Gamma_2$ are of the form $ikLL$ and $ikLR$ ($=kiRL$) where $k\ne i$.  Thus six neighbors of $iL$ lie in $LR\cup RL$. However, these six neighbors are divided evenly between $LR$ and $RL$, since $ikLR$ is in $LR$ if and only if $ikRL$ is in $RL$.
Thus $iL$ has three neighbors in $LR$ and by a symmetric argument the same holds for $jR$.  This gives six more 2-paths starting at $u$ and terminating at some vertex in $RL$.  Thus the total number of such 2-paths  is  $2(16)+ 2(4b_3-b_3^2) +6=38+8b_3-2b_3^{\,2}$ via this second count.

Equating the two counts yields 
$38+b_3=38+8b_3-2b_3^{\,2}$ which simplifies to $2b_3^{\,2}-7b_3=0$.
Clearly, the only integral solution is $b_3=0$ which gives us the following.

\begin{prop}\label{unique OP21}
The unique $G$-orbit partition of $\,\Gamma_2$ where $G\cong Frob(21)$ is as indicated in Figure \ref{final OP21}.
\end{prop}

\begin{figure}
\begin{tikzpicture}
\draw[thick] (0,0) circle (6mm) node at (0,0) {\Large$2$}  node at (0,-.9) {$LL$};
\draw [thick](3,0) circle (6mm) node at (3,0) {\Large$2$}  node at (3,-.9) {$RR$};
\draw [thick](0,3) circle (6mm) node at (0,3) {\Large$0$}  node at (0,3.9) {$LR$};
\draw[thick] (3,3) circle (6mm) node at (3,3) {\Large$0$}  node at (3,3.9) {$RL$};
 
\draw (.6,0)--(2.4,0) node at (1.5,.25) {\large$2$};
\draw (.6,3)--(2.4,3) node at (-.2,1.5) {\large$4$};
\draw (0,.6)--(0,2.4) node at (3.25,1.5) {\large$4$};
\draw (3,.6)--(3,2.4) node at (1.5,3.25) {\large$4$};
\draw (.4,.4)--(2.6,2.6) node at (.68,1.1) {\large$4$};
\draw (.4,2.6)--(2.6,.4) node at (2.32,1.1) {\large$4$};

\end{tikzpicture}
   \caption{The unique $G$-orbit partition of $\Gamma_2$}  \label{final OP21}
\end{figure}
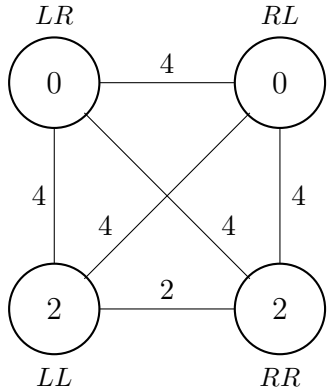

\begin{rema}
Observe that each of $LR\cup\{x\}$ and $RL\cup\{x\}$ is a  coclique of size 22.  This is in fact the largest independence number   allowable by the Hoffman Ratio Bound (aka Hoffman-Delsarte inequality), see \cite{haemers}. Further note from Figure \ref{final OP21} that every vertex outside of $LR\cup\{x\}$ (\resp $RL\cup\{x\}$) has  precisely four neighbors in $LR\cup\{x\}$ (\resp $RL\cup\{x\}$).
\end{rema}

\section{Narrowing the search space}\label{several results}

To this point, we have 
established that if $|G|$ is divisible by 7, then either $G\cong\mathbb Z_7$ or $G\cong Frob(21)$.  At present, we are unable to eliminate either of these two groups in a computer-free manner.   In this section, we gather sufficient structural information about $\Gamma$ that will allow us to prove $G\not\cong Frob(21)$ with
 the aid of a computer.  We would of course welcome independent verification of this fact, and we'd be delighted if a computer-free proof could be furnished. 
 
Let us call a 3-cycle $u_1u_2u_3$ \emph{of type $(X_1Y_1,X_2Y_2,X_3Y_3)$} provided 
$u_i$ is in the orbit $X_iY_i$ where $X_i,Y_i\in\{L,R\}$ for $1\le i\le 3$. 

\begin{prop}\label{interior 3-cycles}
Each of $LL$ and $RR$ consists of seven vertex-disjoint $3$-cycles. 
\end{prop}

\begin{proof}
By transitivity of $G$ on its orbits, each orbit of 3-cycles in $\Gamma_2$ has size 21 with the exception being 3-cycles of type $(LL,LL,LL)$ or $(RR,RR,RR)$. Indeed, orbits of 3-cycles of these two types would have to be of size 7 since the 3-cycles remain internal to their respective orbits.  As there are a total of 140 3-cycles in $\Gamma_2$ (\cf Lemma \ref{140 triangles}), their division into $n_7$ orbits of size 7 and $n_{21}$ orbits of size 21 must satisfy
$7n_7+21n_{21}=140$. Moreover, $0\le n_7\le 2$ since each of $LL$ and $RR$ has internal valency 2.  Since the only integral solution to the above is $n_7=2$ and $n_{21}=6$, each of $LL$ and $RR$ must contain  seven vertex-disjoint  3-cycle.
\end{proof}

For all  $\mathcal O,\mathcal O'\in\{LL,RR,LR,RL\}$, we define 
\[E(\mathcal O,\mathcal O')=\{uv\mid \mbox{$u\in \mathcal O, v\in \mathcal O'$, $u$ and  $v$ are adjacent}\}.\]  
For brevity, we  write $E(\mathcal O)$ rather than the more cumbersome $E(\mathcal O,\mathcal O)$.

\begin{coro}
The subgraph $\Delta$ of $\Gamma_2$ induced on  $E(LL,RR)$ is a disjoint union of cycles. 
\end{coro}

\begin{proof}
This follows at once since $\Delta$ is a bipartite graph of valency 2 (\cf Figure 
\ref{final OP21}).   
\end{proof}
 
We next consider the induced graph $\Gamma_2[\mathcal E]$ where  
\[\mathcal E=\{uw\in E(\Gamma_2)\mid  \mbox{$uvw$ is a 3-cycle 
for some $v\in V(\Gamma_1)$}\}.\]

\begin{prop}\label{red pattern}
\hspace*{-6mm}
\begin{enumerate}

\item[\rm (a)]   $\Gamma_2[\mathcal E]$ is a 2-valent spanning subgraph of $\,\Gamma_2$. 

\item[\rm (b)] 
Every edge in $\mathcal E$ traverses two distinct orbits. Furthermore, $\mathcal E\cap E(LL,RR)=\emptyset$ and $\mathcal E\cap E(LR,RL)=\emptyset$.

\item[\rm (c)]
A cycle in $\Gamma_2[\mathcal E]$ must traverse orbits in the following repetitive order: 
\[LL, LR, RR, RL, LL,LR,RR,RL,\dots, LL, LR, RR, RL, LL.\]

\end{enumerate}
\end{prop}

\begin{proof} 
Recall that every vertex in $\Gamma$ lies on seven 3-cycles. As $iX\in V(\Gamma_1)$ lies on the 3-cycle with vertices  
$x$, $iX$, $iX^C$ where $\{X,X^C\}=\{L,R\}$, each of the remaining six 
3-cycles on $iX$ must have a unique edge in 
 $\mathcal E$. Ranging over all 14 choices for the vertex $iX$, we conclude that $|\mathcal E|= 14\cdot 6=84$.
 
Let  $ijXY$ be an arbitrary  vertex in $\Gamma_2$, and observe that $ijXY$ has $iX$ and $jY$ as its $\Gamma_1$-neighbors. This means the edge adjoining $ijXY$ to $iX$ must lie in a unique 3-cycle with third vertex of the form $i\ell XW$ for some   
 $W\in\{L,R\}$. Likewise,  the edge adjoining $ijXY$ to $jY$ must lie in a unique 3-cycle with third vertex of the form  $jmYU$ for some   
 $U\in\{L,R\}$. 
But then by definition, the edges adjoining $ijXY$ to $i\ell XW$ and $ijXY$ to $jmYU$  must lie in $\mathcal E$. This proves $\Gamma_2[\mathcal E]$ is a spanning subgraph of $\Gamma_2$ and moreover, all 84 vertices in 
$\Gamma_2[\mathcal E]$ must have valency at least 2. 
But the sum of all  
 valencies in $\Gamma_2[\mathcal E]$ must equal $2|\mathcal E|=168$ from which we conclude that $\Gamma_2[\mathcal E]$ is $2$-valent.  Thus (a) is proved.

To prove (b) we first observe that $\mathcal E\cap E(\mathcal O)=\emptyset$ for every orbit $\mathcal O\in \{LL,RR,LR,RL\}$.  Indeed, this follows for    
$\mathcal O\in \{LR,RL\}$ because these orbits are cocliques.  For  
$\mathcal O\in \{LL,RR\}$ the result follows from Proposition \ref{interior 3-cycles} which asserts that every edge in $LL$ or $RR$  already occurs as an edge of an internal  3-cycle.  It follows that every edge in $\mathcal E$ must traverse distinct orbits.  
The proof that $\mathcal E\cap E(LL,RR) =\emptyset$
 is obvious since no two vertices of the form $ijLL$ and $k\ell RR$ can have a common $\Gamma_1$-neighbor. To prove $\mathcal E\cap E(LR,RL)=\emptyset$ we focus on the five 3-cycles in $\Gamma_2$ that occur at a fixed vertex $u$ of $LL$.  Note that $u$ has two neighbors in $RR$ and four neighbors in each of $LR$ and $RL$.  As we have already shown, two of its eight neighbors in $LR\cup RL$ form edges in $\Gamma_2(\mathcal E)$ so cannot lead to 3-cycles at $u$.  As $u$ lies on a unique internal 3-cycle, it is clear that the remaining eight edges incident to $u$ (two into $RR$ and six into $LR\cup RL$) must form edges of the four non-internal 3-cycles at $u$.  
 
However, the two edges into $RR$ cannot form a 3-cycle at $u$ because every edge in $RR$ already occurs in a 3-cycle internal to the orbit $RR$.  
The only remaining possibility is that these four 3-cycles at $u$ are of the form
$(LL,RR,LR)$, $(LL,RR,RL)$, and $(LL,LR,RL)$ (twice).  It is the latter case that holds significance for us.  This is because by transitivity on $LL$, we obtain 42 of the 
84 edges that traverse the orbits $LR$ and $RL$.  Likewise, by replacing $LL$ with $RR$ in the above argument we obtain 42 additional such  edges.    Thus all 84 edges traversing $LR$ and $RL$ are accounted for as sharing a common neighbor in either $LL$ or $RR$, which implies $\mathcal E\cap E(LR,RL)=\emptyset$. This completes  the proof of (b).   

From above, proving (c) is tantamount to showing there are no 2-paths $uvw$ in 
$\Gamma_2[\mathcal E]$ with $u,w\in \mathcal O$ and $v\in \mathcal O'$, where 
$\{\mathcal O,\mathcal O'\}$ is one of  
$\{LL,LR\}$, $\{LL,RL\}$, $\{RR,LR\}$, $\{RR,RL\}$. We illustrate this for
$\{\mathcal O,\mathcal O'\}=\{LL,LR\}$ since the other cases yield to a symmetric argument. 

Suppose first that $\mathcal O=LL$ and $\mathcal O'=LR$, so $u,w\in LL$ and $v\in LR$. By transitivity of $G$ on $LR$, we may assume $v=12LR$. The only possible 
$\Gamma_2[\mathcal E]$-neighbors of $v$ are therefore $u=1jLL$ and $w=1kLL$.  
But in this case $uvw$ is a 3-cycle, which violates $\lambda=1$ since each of $u,v,w$ is adjacent to $1L\in \Gamma_1$. 
Next suppose $\mathcal O=LR$ and $\mathcal O'=LL$ so that $u,w\in LR$ and $v\in LL$.  By transitivity of $G$ on $LR$, there exists $g\in G$ such that $w=u^g$.  But then $w$ is adjacent to $v, v^g\in LL$, thereby reducing this case to the one previously treated.  With this,   the proof of the proposition is complete.
\end{proof}

\begin{rema}
We may now describe the eight orbits of  $3$-cycles in $\Gamma_2$ by type.  Specifically, they are $(LL,LL,LL)$ (one orbit of size $7$),
$(RR,RR,RR)$ (one orbit of size $7$), $(LL,RR,LR)$ (one orbit of size $21$), $(LL,RR,RL)$ (one orbit of size $21$), $(LL, LR, RL)$ (two orbits of size $21$ each), $(RR,LR,RL)$ (two orbits of size $21$ each).
\end{rema}

\begin{coro}\label{red cycles}
The graph $\Gamma_2[\mathcal E]$ is a disjoint union of $\,\frac{84}{k}$ cycles 
of identical length $k$ 
for some $k\in \{4,12,28\}$. 
\end{coro}

\begin{proof}
By Proposition \ref{red pattern}(a), $\Gamma_2[\mathcal E]$ is 2-valent and hence a union of cycles. Moreover, all such cycles must have the same size $k$ by transitivity of $G$ on orbits of vertices.  By Proposition \ref{red pattern}(c),  the edges in such cycles must traverse the orbits 
$LL,LR,RR,RL$ in repetition, which means $k$ is a multiple of 4. One possibility is that the second time the cycle enters the orbit $LL$ it terminates at the initial vertex.
In this case there are 21 $4$-cycles in the graph $\Gamma_2[\mathcal E]$. 
Otherwise, the cycle fails to close after one iteration resulting in a 4-path $uvwzu'$  where $u,u'\in LL$.  But then by transitivity of $G$ on the vertices of $LL$, there exists an element  $g\in G$ for which $u'=u^g$.    Obviously, $|g|=3$ or $7$.  If $|g|=3$ we get seven $12$-cycles, namely the seven images of the 12-cycle 
$uvwzu^gv^gw^gz^gu^{g^2}v^{g^2}w^{g^2}z^{g^2}u$ under the action of $\langle s\rangle$.  A similar result holds for 
$|g|=7$ in which case we obtain three $28$-cycles.  As these are the only possibilities, the proof is complete.
\end{proof}

\begin{rema}
It is currently unclear to us if a Conway $9$-graph $QR(9)$ (aka Paley graph of order $9$)  
exists inside a putative Conway $99$-graph $\Gamma$.  But if this be the case then   $\Gamma_2[\mathcal E]$ would consist of $21$ vertex-disjoint $4$-cycles (\cf Corollary \ref{red cycles}). Moreover, by the transitive action of $G$ on each of its two $\Gamma_1$-orbits, $\Gamma$ would contain $21$ embedded copies of $QR(9)$. 
\end{rema}

\section{Divisibility by $7$ implies $G\cong \mathbb Z_7$}\label{computer}
	
Up to this point, we have shown that if $|G|$ is divisible by 7 then 
there are just  two possibilities for the  isomorphism type of $G$, namely $\mathbb Z_7$ and $Frob(21)$. In Section \ref{cyclic or Frobenius} we determined the unique orbit partition of $\Gamma$ under the assumption that $G\cong Frob(21)$, while in Section \ref{several results}  we derived a fairly comprehensive structural framework for $\Gamma$ under this same assumption. 
The framework obtained 
is essential in reducing run-time,  thereby making a computer search feasible. 
Below we give details of our computer-generated proof that $G$ is isomorphic to 
$\mathbb Z_7$ provided $|G|$ is divisible by 7.

The program was written and implemented in GAP  Version 4.12.2 \cite{GAP}.  Specifically, we utilized the GRAPE package Version 4.9.0 \cite{GRAPE}. 
As the computer code is prohibitively large we cannot reproduce it here, however it is available upon request.  

Below we sketch the key steps of the program.\\[2mm]
\noindent{\bf Step 1:}
Initialize the vertex set, then create edges common to all cases of the search, namely those incident to each vertex in $\Gamma_1$. These account for all edges from $x$ to its $\Gamma_1$-neighbors, all edges internal to $\Gamma_1$, and all edges that traverse $\Gamma_1$ to $\Gamma_2$.\\[1mm]
\noindent{\bf Step 2:}
Form the seven 3-cycles internal to $LL$ by applying the AddEdgeOrbit function to a chosen internal edge incident to $12LL$.\\[1mm]
\noindent{\bf Step 3:}
Repeat step 2 to generate the seven 3-cycles in the orbit $RR$.\\[1mm] 
\noindent{\bf Step 4:}
Choose   
six possible neighbors of $12LL$ in the orbit $LR$, followed by three possible choices for such neighbors in $RL$.  Then apply the AddEdgeOrbit function to obtain half of the edges of the induced graph $\Gamma_2[\mathcal E]$.\\[1mm]
\noindent{\bf Step 5:}
Repeat step 4 at the vertex $12RR$, thus determining the entire  graph $\Gamma_2[\mathcal E]$.\\[1mm] 
\noindent{\bf Step 6:}
Choose the eight remaining neighbors of $12LL$. (This consists of three neighbors in $LR$, three in $RL$ and two in $RR$.) Apply the AddEdgeOrbit function to determine all edges incident to $v$ for every vertex in $v\in LL$. \\[1mm]
\noindent{\bf Step 7:}
Repeat step 6  for the six remaining neighbors of $12RR$. (This consists of three neighbors in $LR$ and three in $RL$.) This determine all edges incident to $v$ for every vertex in $v\in RR$. \\[1mm]
\noindent{\bf Step 8:}
Complete the graph by  forming all edges traversing the orbits $LR$ and $RL$.  This is accomplished by choosing four possible  neighbors of $12LR$ and applying the AddEdgeOrbit function. The graph is now ready for testing.\\[1mm]
\noindent{\bf Step 9:}  
Compare $[[ 0, 0, 14 ], [ 1, 1,12 ], [ 2, 12, 0 ]]$ to the output yielded by  the GlobalParameters function.

The program ran through $2,916$ iterations. Each iteration was comprised of multiple cases ranging between the hundreds and thousands. We were able to avoid cases ranging well into the millions  by continually exploiting the DoubleCosetRepsAndSizes function. This function was used when it was theoretically evident that every
 element in the double coset $HgK$ would produce the same computational result as $g$, so need not be tested. This occurred at steps $6,7$ and $8$ where $H$ and $K$ were suitably chosen ``coordinate'' stabilizers in the symmetric groups $S_{10}$, $S_6$ and $S_4$, respectively.

\bibliographystyle{amsplain-ac}
\bibliography{myBib}
\end{document}